\documentclass[a4paper,leqno,12pt]{amsart}
\usepackage{pb-diagram}
\usepackage{epsf,graphicx,epsfig}
\usepackage{amscd}
\usepackage{amsmath,latexsym,amssymb,amsthm}
\usepackage{setspace,cite}
\usepackage{lscape,fancyhdr,fancybox}
\usepackage{graphicx}

\usepackage{color}
\usepackage{url}
\usepackage{enumerate}
\usepackage[mathscr]{euscript}
\input xy
\xyoption{all}

 \theoremstyle{plain}

\swapnumbers
    \newtheorem{thm}{Theorem}[section]
    \newtheorem{prop}[thm]{Proposition}
    \newtheorem{corollary}[thm]{Corollary}
    
    \newtheorem{subsec}[thm]{}
\theoremstyle{definition}
    \newtheorem{defn}[thm]{Definition}
    \newtheorem{exam}[thm]{Examples}
    
      \newtheorem{lemma}[thm]{Lemma}

\theoremstyle{remark}
     \newtheorem{remark}[thm]{Remark}

    \newtheorem{ack}[thm]{Acknowledgements}

\newenvironment{mysubsection}[2][]
{\begin{subsec}\begin{upshape}\begin{bfseries}{#2.}
\end{bfseries}{#1}}
{\end{upshape}\end{subsec}}

\newenvironment{myeq}[1][]
{\stepcounter{thm}\begin{equation}\tag{\thethm}{#1}}
{\end{equation}}

\newcommand{\Ab}{\mathcal{A}b}
\newcommand{\Grp}{\mathcal{G}rp}

\newcommand{\R}{{\mathbb R}}

\newcommand{\ints}{{\mathbb{Z}}}

\newcommand{\ha}{\hat{a}}

\newcommand{\upi}{\underline{\pi}}

\newcommand{\uphi}{\underline{\phi}}

\newcommand{\OG}{\mathcal{O}_G}
\newcommand{\OGop}{\OG^{\op}}
\newcommand{\Crs}{{\EuScript Crs}}
\newcommand{\CRS}{{\EuScript CRS}}
\newcommand{\Top}{{\EuScript Top}}
\newcommand{\Set}{{\EuScript Set}}
\newcommand{\FTop}{{\EuScript {FT}op}}

\newcommand{\uc}[1]{\chi_G(\MM,\,{#1})}
\newcommand{\realz}[1]{|{#1}|}
\newcommand{\uch}[1]{\chi_{\uphi(G/H)}(M(G/H),\,{#1})}

\newcommand\Ss{{\mathcal S}}

\newcommand\A{{\mathcal A}}
\newcommand\BB{{\mathfrak B}}
\newcommand\CC{{\mathcal C}}
\newcommand\DD{{\mathcal D}}
\newcommand\EE{{\mathcal E}}
\newcommand\FF{{\mathcal F}}

\newcommand\HH{{\mathcal H}}

\newcommand\LL{{\mathcal L}}
\newcommand\MM{{\mathcal M}}

\newcommand\PP{{\mathcal P}}

\newcommand\TT{{\mathcal T}}
\newcommand\UU{{\mathcal U}}

\newcommand\PMF{{\PP\kern-2pt\MM\FF}}

\newcommand\PML{{\PP\kern-2pt\MM\LL}}

\newcommand{\fsubd}{\mathrel{{\scriptstyle\searrow}\kern-1ex^d\kern0.5ex}}
\newcommand{\bsubd}{\mathrel{{\scriptstyle\swarrow}\kern-1.6ex^d\kern0.8ex}}
\newcommand{\fsubeq}{\mathrel{\raise-.7ex\hbox{$\overset{\searrow}{=}$}}}
\newcommand{\bsubeq}{\mathrel{\raise-.7ex\hbox{$\overset{\swarrow}{=}$}}}

\newcommand{\tsh}[1]{\left\{\kern-.9ex\left\{#1\right\}\kern-.9ex\right\}}

\newcommand{\Fibre}{\operatorname{Fibre}}
\newcommand{\id}{\operatorname{id}}
\newcommand{\pr}{\operatorname{pr}}
\newcommand{\CW}{\operatorname{CW}}

\newcommand{\Coh}{\operatorname{Coh}}

\newcommand{\colim}{\operatorname{colim}}
\newcommand{\diag}{\operatorname{diag}}
\newcommand{\Sing}{\operatorname{Sing}}
\newcommand{\Maps}{\operatorname{Map}}
\newcommand{\Hom}{\operatorname{Hom}}

\newcommand{\ob}{\operatorname{Ob}}
\newcommand{\op}{\operatorname{op}}
\begin{document}

\title{Representing Bredon cohomology with local coefficients}

 \author{Samik Basu}
\email{samik.basu2@gmail.com; samik@rkmvu.ac.in}
\address{Department of Mathematics,
 Vivekananda University,
 Belur, Howrah 711202,
West Bengal, India.}

 \author{Debasis Sen}
 \thanks{The second author was supported by INSPIRE research grant.}
\email{debasis@iitk.ac.in}
\address{Department of Mathematics and Statistics,
 IIT Kanpur,
 India.}
\date{\today}

\begin{abstract}
For a discrete group $G$, we represent the Bredon cohomology with local coefficients as the homotopy classes of maps in the category of equivariant crossed complexes. Subsequently, we construct a naive parametrized $G$-spectrum, such that the cohomology theory defined by it reduces to the Bredon cohomology with local coefficients when restricted to suspension spectra.
\end{abstract}

\maketitle

\setcounter{section}{0}

\section{Introduction}
Let $G$ be a discrete group. In \cite{mm,ms} the authors introduced the notion of Bredon cohomology with local coefficients and constructed a representing $G$-space. In this paper we continue the study of this representability result.

In \cite{git}, Gitler proved that the cohomology groups of a space with local coefficients are representable in the homotopy category. The classifying space for the $n$-th cohomology group $H^n(X;\A)$ of $X$ with local coefficients $\A$ is the generalized Eilenberg-Mac~Lane complex $L_{\pi}(A,n)$, where $\A$ is given by an action of $\pi=\pi_1(X)$ on $A$. The space $L_{\pi}(A,q)$ appears as the total space of a fibration $L_{\pi}(A,q) \rightarrow  K(\pi,1)$. The fibration may be interpreted as an object of the slice category $\Top/K(\pi,1)$, where $\Top$ denotes the category of topological spaces and continuous maps. There is a canonical map $X \rightarrow K(\pi,1)$, inducing the identity on fundamental group, so that $X$ can be viewed as in $\Top/K(\pi,1)$. The classification theorem states that $H^n(X;\A)$ is isomorphic to the homotopy classes of maps $[X,L_{\pi}(A,n)]_{\Top/K(\pi,1)}$ in the slice category $\Top/K(\pi,1)$.

This result was extended to Bredon cohomology with local coefficients in \cite{m,ms}. In this case, the representing space can be written using the construction of Elmendorf (see \cite{elm}) which establishes an equivalence between the homotopy category of $G$-spaces and contravariant functors from the orbit category $\OG$ to spaces. The main idea there was to use Gitler's result to construct the fixed points for each subgroup $H$ of $G$ and then use Elmendorf's construction to form the representing $G$-space.

In this paper we write down this representability result in two ways, using crossed complexes and parametrized spectra respectively. Crossed complexes encode the algebraic properties of the sequence $\{\pi_n(X_n,X_{n-1},v)_{v\in X_0}\}_{n\geq 1}$ associated to a skeletal filtration $X_0\subset X_1\subset \cdots$ of a CW complex $X$. The idea of reduced crossed complex, i.e. when $X_0=*$, have  been
studied by Blakers \cite{bla}, J.H.C. Whitehead \cite{white1} and Huebschmann \cite{hue1,hue2} under the names `group systems', `homotopy systems' and `crossed  resolutions', respectively. It was defined in full generality by Brown and Higgins in \cite{rbrown5, rbrown5a} and studied further in \cite{rbrown3, rbrown2, rbrown, brownb}. It is has been shown in \cite[Proposition $4.9$]{rbrown3} that cohomology with local coefficients can be represented as homotopy classes of maps in crossed complexes. We extend this representability result for Bredon cohomology with local coefficients (cf. \cite{m,ms}) using equivariant crossed complexes, which is identified as $\OG\mbox{-}\Crs$, the diagram category of contravariant functors from $\OG$ to crossed complexes (see \cite{rbrown2}). In this regard, we construct a series of equivariant crossed complexes $\chi_G(\MM ,n) \rightarrow \chi_G(\upi,1)$ arising from a $\upi$-module $\MM=(M,\uphi)$. Here $\upi$ is an $\OG$-group (cf. Definition \ref{oggrp}). Combining \cite[Theorem 3.3]{m} and \cite[Theorem 4.1]{brownb}, we deduce the following.
\newtheorem*{thma}{Theorem A}
\begin{thma}
Given a $\upi$-module $\MM=(M,\uphi)$ and $G$-CW complex $X$ together with a map $\theta\colon X\to K_G(\upi,1),$ the $n$-th Bredon cohomology with local coefficients $H^n_G(X;\theta^*\MM)$ is isomorphic to $\pi_0 \Coh\underline{\Crs}(\Pi_G(X),\uc{n})_{\chi_G(\upi,1)}.$
\end{thma}
\noindent  Here $\Coh\underline{\Crs}(\Pi_G(X),\uc{n})_{\chi_G(\upi,1)}$ is the simplicial set of homotopy coherent transformations in the over category $\OG\mbox{-}\Crs/\chi_G(\upi,1)$ and $K_G(\upi,1)$ is the equivariant Eilenberg-Mac~Lane space. See Theorem \ref{thmfirst} below.

The idea of representing cohomology theories is best achieved by spectra in stable homotopy theory: the various topological spaces representing the different cohomology groups fit together to form a spectrum, and the cohomology theory is represented by an object in the stable homotopy category which is the homotopy category of spectra in an appropriate model category structure. In \cite{par}, May and Sigurdsson  have defined parametrized spectra over a space $B$ and showed that these represent cohomology theories in the category $\Top/B$.  Fix a group $\pi$. In the case of cohomology with local coefficients, the domain category consists of CW complexes $X$ over $K(\pi,1)$. We show that cohomology with local coefficients can be represented by a parametrized spectrum, constructed using classifying spaces of the representing crossed complexes (see \cite{rbrown3}). 

Lastly, for a discrete group $G$, we combine the results of the non equivariant case using the `coalescence functor' $\Psi$ of Elmendorf \cite{elm}. We construct an equivariant parametrized spectrum $J_G\MM$ associated to an equivariant local coefficient system $\MM$, again by using the geometric realization of the nerve of the equivariant crossed complexes $\uc{n}$. $J_G\MM$ is a naive $G$-spectrum (indexed over a trivial $G$-universe). This is expected for arbitrary coefficient systems, since one lacks the required transfer maps to index the spectrum over a complete $G$-universe.
\newtheorem*{thmb}{Theorem B}
\begin{thmb}
For $X$ and $\MM$ as in Theorem A, the $n$-th Bredon cohomology with local coefficients $H^n_G(X;\theta^*\MM)$ is isomorphic to the set of maps of parametrized spectra over the equivariant Eilenberg-Mac~Lane $G$-space $K_G(\upi,1)$ written as
 $[\Sigma^\infty_{K_G(\upi,1)} X_{+ K_G(\upi,1)},\Sigma_{K_G(\upi,1)}^n J_G \MM]_{K_G(\upi,1)}$.
\end{thmb}
\noindent (See Theorem \ref{thm:tgspec} below.)

It is important to note that the use of classifying spaces of equivariant crossed complexes leads to an easy and explicit description of the parametrized spectrum in the above theorem. The assignment of a parametrized spectrum to an equivariant local coefficient system is clearly functorial.     

\begin{mysubsection}{Organization}
The paper is organized as follows. In Section \ref{scrscomplex}, we review some preliminaries on crossed complexes and the representation of cohomology with local coefficients in the category of crossed complexes. In section 3, we recall the definition of Bredon cohomology with local coefficients and prove a suspension isomorphism in this context. Using these results we represent Bredon cohomology with local coefficients using equivariant crossed completes. In Section \ref{slocparspc}, we construct a parametrized $\Omega$-spectrum $J_\pi(A)$ and show that the associated cohomology theory is cohomology with local coefficients. Finally in Section \ref{squipar} we construct a parameterized equivariant $\Omega$-spectrum $J_G\MM$ and prove an equivariant analogue.
\end{mysubsection}
\ack
We thank Prof. Goutam Mukherjee for many helpful discussions on this topic. We also thank Prof. David Blanc for taking the time to go through a draft of this paper and give a number of helpful suggestions to improve its presentation. We are grateful to Prof. Ronnie Brown for many useful suggestions for improvement in the article.

\section{Crossed complexes and cohomology with local coefficients}\label{scrscomplex}
In this section we recall basic definitions of crossed complexes and its relation with cohomology with local coefficients. 

\begin{mysubsection}{Crossed complexes}
Recall from \cite{rbrown3} that a crossed complex is a chain complex of modules over a groupoid, with possibly non-abelian automorphism groups of objects in degrees one and two. The notion of a crossed complex encodes the algebraic properties of the sequence $\lbrace \pi_n(X_{n},X_{n-1},v)_{v\in X_0}\rbrace_{n=1}^{\infty}$ associated to a filtration $X_{0} \subset X_{1} \subset \ldots$ of a topological space $X$.

A groupoid $\mathcal{G}$ is said to act on a totally disconnected groupoid $\mathcal{N}$ with same object set as $\mathcal{G}$, i.e., 
$\mathcal{N}=\sqcup_{x\in \ob(\mathcal{G})} N_x$ if there are action maps
$$N_x \times \mathcal{G} (x,y) \to N_y,~~x,y\in \ob(\mathcal{G}),$$
that respect the groupoid composition and the group structure in an obvious way.
If $\mathcal{N}$ is abelian then it is called an $\mathcal{G}$-module. 

%
\begin{defn} A \emph{crossed complex} $\CC$ consists of a set $C_0$, and groupoids $C_n$ for $n\geq 1$ with the same object set $C_0$. There are morphisms of groupoids $\delta_n\colon C_n \rightarrow C_{n-1}$ and an action of the groupoid $C_1$ on $C_n$ for $n\geq 2$. This satisfies the following conditions:
\begin{enumerate}[(i)]
\item For $n\geq 2$, $C_n$ is a totally disconnected groupoid, that is, $C_n$ is a family of groups $\lbrace C_n(v): v\in C_0=\ob(C_n)\rbrace$. The groups $C_n(v)$ are abelian if $n\geq 3$.
\item For $n\geq 2$, we denote the action $C_n \times C_1 \rightarrow C_n$ by $(c,c_1)\mapsto c^{c_1}$.  The action of $C_1$ on itself is by conjugation.
\item For $n\geq 2$ the morphism $\delta_n$ is the identity on the set of objects and commutes with the action of $C_1$. For $n\geq 3$, the composite $\delta_{n-1} \circ \delta_n$ is the map which takes all the morphisms to the identity. The morphisms in the image of $\delta_2$ act trivially on $C_n$ for $n\geq 3$ and by conjugation on $C_2$.
\end{enumerate}
\end{defn}
We write $s,t\colon C_1\rightarrow C_0$ for the source and target map of the groupoid $C_1$, and $t\colon C_n\rightarrow C_0,~n\geq 2,$ denotes the target, or the base point map, of the totally disconnected groupoid $C_n.$
\begin{defn}
 A \emph{morphism of crossed complexes} $f\colon \CC\rightarrow \DD$ is a family of morphisms of groupoids $f_n\colon C_n\rightarrow D_n,~n\geq 1,$ all inducing the same map of objects $f_0\colon C_0\rightarrow D_0$, such
that $\delta_n f_n(c) = f_{n-1}\delta_n(c)$ and $f_n(c^{c_1} ) = f_n(c)^{f_1(c_1)}$ for all $c\in C_n, c_1\in C_1,~n\geq 1$.
\end{defn}
We denote the category of crossed complexes by $\Crs$.
\begin{exam}
 Let $\{X_n\}_{n\in \mathbb{N}}$ be a filtration of a topological space $X$. Then, we have a crossed complex $\Pi(X)$, defined as:
 $$
\Pi(X)_n=
 \left\{ \begin{array}{lr}
 X_0 &\mbox{ if $n=0$} \\
 \pi_1(X_1,X_0) &\mbox{if $n=1$} \\
  \pi_n(X_n,X_{n-1},X_0) &\mbox{if $n>1$}
\end{array} \right.
$$
Here $\pi_1(X_1,X_0)$ is the fundamental groupoid of $X_1$ on the set $X_0$ of points,
and $\pi_n(X_n , X_{n-1},X_0)$ is the family of relative homotopy groups $\pi_n(X_n,X_{n-1},v)$
for all $v\in X_0$. The group $\pi_1(X_1,v)$ operates in the usual way on $\pi_n(X_n,X_{n-1},v)$, which gives the $\pi_1(X_1,X_0)$ action on $\pi_n(X_n , X_{n-1},X_0)$. The differential $\delta_n,~n\geq 3,$ is the composite $$\pi_n(X_n,X_{n-1},v)\rightarrow \pi_{n-1}(X_{n-1},v)\rightarrow \pi_{n-1}(X_{n-1},X_{n-2},v),$$ where the first map is the usual boundary of the pair $(X_n,X_{n-1})$ and the second map is induced by inclusion. The other differential $\delta_2\colon \pi_2(X_2,X_1,v)\rightarrow \pi_1(X_1,v)$ is the standard boundary map of the pair $(X_2,X_1)$. With these structures, $\Pi(X)$ is a crossed complex and called the \emph{fundamental crossed complex} of the filtered topological space $X$.
\end{exam}
If $X_0=*$, then we simply write $\Pi(X)_n=\pi_n(X_n,X_{n-1})$. The category of filtered spaces, denoted by $\FTop$, has objects filtered topological spaces and morphisms continuous maps which respect the filtration. Then $\Pi$ is a functor from $\FTop$ to $\Crs$. For a CW complex $X$, the associated crossed complex is given by the skeleton filtration.

\begin{defn}\cite[Definition 7.1.38]{brownb}\label{homotopy}
 Let $f,g\colon\CC\rightarrow \DD $ be maps between two crossed complexes. A \emph{homotopy} from $f$ to $g$, written as $\HH\colon f\sim g$, is a sequence of maps $\HH_n\colon C_n\rightarrow D_{n+1},~n\geq 0$ which satisfy the following properties:
\begin{enumerate}[(a)]
\item For $c\in C_n,~n\geq 0$,
$$\HH_n(c)\in
 \left\{ \begin{array}{rl}
  D_1(f_0(c),g_0(c))&\mbox{ if $n=0$} \\
  D_{n+1}(tg_n(c))&\mbox{ if $n\geq 1$}
       \end{array} \right.
       $$
\item If $c,c^{'}\in C_n$ and $c c^{'}$ or $c+c^{'}$ is defined according as $n=1$ or $n\geq 2$,
then
\begin{eqnarray*}
&& \HH_1(c c^{'}) =\HH_1(c)^{g_1(c)}\HH_1(c^{'}) \\
&& \HH_n(c+c^{'})=\HH_n(c)+\HH_n(c^{'})~~ \mbox{if}~~ n\geq 2.
\end{eqnarray*}
\item
For $n\geq 2$, $\HH_n$ preserves the action over $g$, i.e., if $c\in C_n,~n\geq 2,~c_1\in C_1$, and $c^{c_1}$ is defined, then
$
 \HH_n(c^{c_1}) = \HH_n(c)^{g_1(c_1)}.
$
\item
The pair $(\HH,g)$ determines the initial morphism $f$; if $c\in C_n,~n\geq 0$ then:
$$f_n(c)=
 \left\{ \begin{array}{lr}
 s\HH_0(c)&\mbox{ if $n=0$} \\
 \HH_0(sc)g_1(c)\delta_2 \HH_1(c) \HH_0(tc)^{-1}&\mbox{ if $n=1$} \\
  \{g_n(c)+\HH_{n-1}\delta_n(c) +\delta_{n+1}\HH_n(c)\}^{\HH_0(tc)^{-1}}&\mbox{ if $n\geq 2$}
       \end{array} \right.
       $$
\end{enumerate}
\end{defn}
{\bf Note}: The axiom (d) in the above definition reflects the condition $f-g = dH + Hd$ for a homotopy $H$ between maps $f,g$ of chain complexes.
 
The category $\Crs$ has a model category structure (see \cite[Theorem 2.12]{rbrown4}) and the `homotopy' relation is an equivalence relation on the set of morphisms of crossed complexes (see \cite{qui}). We denote the homotopy classes of maps from $\CC$ to $\DD$ by $[\CC,\DD].$

The category $\Crs$ of crossed complexes has internal hom of maps $\CRS(\CC,\DD)$ (cf. \cite[Theorem 9.3.6]{brownb}). The groupoid $\CRS_0(\CC,\DD)$ is defined to be the set $\Crs(\CC,\DD)$ of crossed complex maps, $\CRS_1(\CC,\DD)(f,g)$ is the set of homotopies from $f$ to $g$ and $\CRS_m(\CC,\DD)(f)$ is the set of $m$-fold homotopies from $\CC$ to $\DD$ over $f$, defined as follows.
\begin{defn}\cite[Section 9.3.i]{brownb}\label{mhomotopy}
Let $m\geq 2$. An \emph{$m$-fold homotopy} $\HH$ from $\CC$ to $\DD$ over a morphism $f\colon \CC\rightarrow\DD$, is given by maps $\HH_n\colon C_n\rightarrow D_{n+m}$ for each $n\geq 0$ which satisfy:
\begin{enumerate}[(a)]
\item For $n\geq 2$, if $c\in C_n$ and $c_1 \in C_1$,
$$\HH_n(c^{c_1}) = \HH_n(c)^{f_1(c_1)}$$
\item $\HH_1$ is a derivation over $f$, i.e. if $c,c'\in C_n$ and $c+c'$ is defined then
$$\HH_1(c+c')= \HH_1(c)^{f_1(c')}+ \HH_1(c')$$
\item For $n\geq 2$ the $\HH_n$ are morphisms, i.e. if $c,c' \in C_n$ and $c+c'$ is defined then
$$\HH_n(c+c')=\HH_n(c)+\HH_n(c').$$
\end{enumerate}
\end{defn}
Furthermore $\Crs$ is a complete and cocomplete simplicially enriched category in which all the `hom-sets' are Kan complexes (See \cite{rbrown2}).

We shall denote the category of simplicial sets and simplicial maps by $\Ss$. Given a simplicial set $K\in \Ss$ the geometric realization $\realz{K}$ of $K$ has a natural filtration of skeleta. Composing further with the functor $\Pi$ we get a functor from $\Ss$ to $\Crs$. This is the fundamental crossed complex functor of simplicial sets which, by abuse of notation, we also denote by $\Pi$.

It is proved in \cite{rbrown3} that the functor $\Pi$ has a right adjoint and hence it preserves colimits. The right adjoint is defined using the nerve of a crossed complex.

\begin{defn}\label{dnerve}
 The \emph{nerve functor} $N\colon \Crs\rightarrow \Ss$ is defined by $$N^{\Delta}(\CC)_n:=\Crs(\Pi(\Delta^n),\CC),$$ where $\Delta^n$ is the standard topological $n$-simplex with standard cell structure and cellular filtration. The simplicial maps of $N^{\Delta}(\CC)$ are induced by the face and degeneracy maps of $\Delta^n$.

The \emph{classifying space} $\BB(\CC)$ of a crossed complex $\CC$ is defined as $\realz{N^{\Delta}(\CC)}$, the geometric realization of the simplicial set $N^{\Delta}(\CC)$.
\end{defn}
We have the following result.
\begin{thm}[Theorem A, \cite{rbrown3}]\label{thadjoint}
 If $X$ is a CW complex, and $\CC$ is a crossed complex, then there is a weak
homotopy equivalence
$$\eta\colon \BB(\CRS(\Pi(X),\CC))\rightarrow \Maps_{\Top}(X,\BB(\CC)),$$
and a bijection of sets of homotopy classes
\begin{myeq}\label{eqadjoncrs}
[\Pi(X),\CC]\cong [X,\BB(\CC)] 
\end{myeq}
which is natural with respect to morphisms of $\CC$ and cellular maps of $X$.
\end{thm}

We can encode the entire information above in the following diagram of categories and adjoint functors.
$$\xymatrix{ \Top \ar@<-0.5ex>[rrr]_{\Sing} &&& \Ss \ar@<0.5ex>[d]^{\Pi} \ar@<0.5ex>[dlll]^{\realz{-}} \ar@<-0.5 ex>[lll]_{\realz{-}} \\
             \FTop\ar[rrr]_{\Pi} \ar[u]^{\colim}&&& \Crs \ar@<0.5ex>[u]^{N^{\Delta}} }.$$

\end{mysubsection}

\begin{mysubsection}{Cohomology with local coefficients and crossed complexes}\label{sloccrs}
In this section we recall the definition of cohomology with local coefficients and its representability result in the category of crossed complexes (see \cite{rbrown3}).
\begin{mysubsection}{Local coefficient system}\label{cohloc}
Let $X$ be a topological space. A \emph{local coefficient system} on $X$ is a contravariant functor from $\A\colon \pi X\rightarrow \Ab$, where $\Ab$ denotes the category of abelian groups. 
Recall that the fundamental groupoid $\pi X$ is a category whose objects are points of $X$ and morphism from $v\in X$ to $w\in X$ is the set of homotopy classes of paths from $v$ to $w$.
A continuous map $f\colon X\rightarrow Y$ induces a functor $\pi X\rightarrow \pi Y$ and hence a local coefficient system $\A$ on $Y$ induces a local coefficient system $f^*(\A)$ on $X$.

Suppose $X$ is path-connected. Then a local coefficient system on $X$ is equivalent to $\pi_1(X,v)$-module, for a chosen point $v \in X$.
\end{mysubsection}

\begin{mysubsection}{Cohomology with local coefficients}\label{lc}
Let $\A$ be a local coefficient system on a topological space $X$. Denote by $C^{n}(X;\A)$ the group of all functions $f$ defined on singular $n$-simplices $\sigma\colon \Delta^n\rightarrow X$ such that $f(\sigma)\in \A(\sigma(0))$. Define a homomorphism $$\delta\colon C^{n-1}(X;\A)\rightarrow C^{n}(X;\A),~f\mapsto \delta f,$$ by
$\delta f(\sigma)=\A(\sigma|_{[0,1]})f(\sigma^{(0)})+\sum_{j=1}^{n}(-1)^{j}f(\sigma^{(j)})$, where $\sigma^{(j)}$ denotes the $j$-th face of $\sigma$. Then $\{C^{*}(X;\A),\delta \}$ is a cochain complex.
\begin{defn}
Let $\A$ be a local coefficient system on a topological space $X$. Then the $n$-th cohomology of $X$ with local coefficients $\A$ is defined by $$H^n(X;\A):=H^n(\{C^*(X;\A),\delta\}).$$
\end{defn}
\end{mysubsection}

\begin{mysubsection}{Representing cohomology with local coefficients}\label{ucs}
Let $\pi$ be a group and $(A,\phi)$ be a $\pi$-module. Following \cite[Definition 7.1.11]{brownb}, we first define representing crossed complexes $\chi_{\phi}(A,n),~ \chi(\pi,n)$ together with maps of crossed complexes $p\colon \chi_{\phi}(A,n)\rightarrow \chi(\pi,1)$ for each $n\geq 0$.

\begin{defn}
For a group $\pi$, the crossed complexes $\chi(\pi ,n),~n\geq 1,$ are defined by the formula:
$$\chi(\pi,n)_m=
 \left\{ \begin{array}{rl}
 \pi &\mbox{ if $m=n$} \\
  * &\mbox{ otherwise}
       \end{array} \right.
       $$
The $\delta_i$'s and the action of $ \chi(\pi,n)_1$ on $\chi(\pi,n)_m$ are obvious.

The crossed complex $\chi(\pi,0)$ is defined by
 $$\chi(\pi,0)_m=
 \left\{ \begin{array}{rl}
 U(\pi) &\mbox{ if $m=0$} \\
  \lbrace \id_x | x\in \pi\rbrace &\mbox{ otherwise}
       \end{array} \right.
       $$
   Here $U(\pi)$ denote the underlying set of $\pi$.
 The maps $\delta_i$ and the actions of $\chi(\pi,0)_1$ are automatically fixed.
\end{defn}
 The classifying space $\BB\chi(\pi,n) = \realz{N^\Delta (\chi(\pi,n))}$ is the Eilenberg-Mac~Lane space $K(\pi,n)$.

\begin{defn}
For a $\pi$-module $(A,\phi)$, we define a fibration (\cite{rbrown3}) of crossed complexes $p\colon \chi_{\phi}(A,n)\to \chi(\pi,1),~n\geq 0$ as follows:
\begin{itemize}
\item For $n\geq 2$ the crossed complex $\chi_{\phi}(A,n)$ is defined by:
 $$ \chi_{\phi}(A,n)_m=
 \left\{ \begin{array}{rl}
 \pi &\mbox{ if $m=1$} \\
 A &\mbox{if $m=n$} \\
  * &\mbox{ otherwise}
       \end{array} \right.
       $$
The $\delta_i$ are all trivial and the action of $\pi$ on $A$ is given by $\phi$. Note that we have a canonical map $p\colon\chi_{\phi}(A,n)\rightarrow \chi(\pi,1)$, which is the identity on the first level and trivial on all the other levels.

\item The crossed complex $\chi_{\phi}(A,1)$ is defined using the equation
 $$ \chi_{\phi}(A,1)_m=
 \left\{ \begin{array}{rl}
 \pi\ltimes A &\mbox{ if $m=1$} \\
  * &\mbox{ otherwise}
       \end{array} \right.
       $$
In this case the map $p\colon \chi_{\phi}(A,1) \rightarrow \chi(\pi,1)$ is the projection at the first level.

\item To define $\chi_{\phi}(\pi,0)$, we first note the definition of $\mathit{Gpd}(G,M)$, the translation  groupoid associated to a group action of $G$ on a set $M$. It has $M$ as the set of objects and $g\in G$ is a morphism from $m\in M$ to $m'\in M$ if $gm=m'$. There is a functor from $\mathit{Gpd}(G,M)$ to $G$ (considered as groupoid with one object). The crossed complex $\chi_{\phi}(A, 0)$ is defined by
$$\chi_{\phi}(A,0)_m= \left\{ \begin{array}{rl}
 A &\mbox{ if $m=0$} \\
 \mathit{Gpd}(\pi,A) &\mbox{ if $m=1$}\\
 \lbrace \id_x | x\in A\rbrace &\mbox{ otherwise}
       \end{array} \right.
       $$
 The maps $\delta_i$ and the action is automatically fixed. The map to $\chi(\pi,1)$ is trivial on all dimensions except at 1 where it is the map between groupoids described above.
\end{itemize}
\end{defn}
The classifying space $\BB(\chi_{\phi}(A,n))=\realz{N^{\Delta}(\chi_{\phi}(A,n))}$ is the generalized Eilenberg-Mac~Lane space $L_{\pi}(A,n),$ which is the representing space for cohomology with local coefficients (see \cite{git, hir,bfg,gj}).

Let $X$ be a reduced CW complex and $\alpha\colon \pi_1(X,*)\rightarrow \pi$ be a group homomorphism. We can view $A$ as a $\pi_1(X,\ast)$-module via $\alpha$ and hence $\alpha$ determines a local coefficient system $\A_{\phi\alpha}$ on $X$.
\begin{thm}\label{thbprop}\cite[Proposition 4.9]{rbrown3}
 With $X$ as above, $H^n(X;\A_{\phi\alpha})$ is isomorphic to $[\Pi(X),\chi_{\phi}(A,n)]_{\alpha}$ for $n\geq 2$, where $[\Pi(X),\chi_{\phi}(A,n)]_{\alpha}$ denotes the homotopy classes of maps from $\Pi(X)$ to $\chi_{\phi}(A,n)$ inducing $\alpha$ on fundamental groups.
\end{thm}
\end{mysubsection}
\end{mysubsection}

\section{ Bredon cohomology with local coefficients and equivariant crossed complexes}\label{seqtheo}
In this section we review the definition of Bredon cohomology with local coefficients and write down a suspension isomorphism for it. Using a result of M{\o}ller (\cite{m}), we formulate an equivariant version of Theorem \ref{thbprop}. 

\begin{mysubsection}{Bredon cohomology with local coefficients}\label{sbre}
We first recall the \emph{orbit category} $\OG$ of a discrete group $G$ and the associated notions.

The objects of the category $\OG$ are the coset spaces $G/H=\{gH|~g\in G\}$, as $H$ runs over the all subgroups of $G$. The group $G$ acts on the set $G/H$ by left translation. A morphism from $G/H$ to $G/ K$ is a $G$-map. Note that a subconjugacy relation $a^{-1}Ha\subseteq K,~a\in G,$ determines a $G$-map $\ha\colon G/H\rightarrow G/K$, given by $\hat{a}(eH)=aK$. Conversely, any $G$-map from $G/H$ to $G/K$ is of this form (cf. \cite{br}).
\begin{defn}\label{oggrp}
A functor $\upi$ from $\OGop$ to the category $\Grp$ of groups is called an \emph{$\OG$-group}. A map between $\OG$-groups is a natural transformation of functors.
\end{defn}
 The category of $\OG$-groups is denoted by $\OG$-$\Grp$.
The notion of \emph{$\OG$-space} or \emph{abelian $\OG$-group} has the obvious meaning replacing $\Grp$ by $\Top$ or $\Ab$, the category of abelian groups. Similarly, we can talk of \emph{$\OG$-crossed complexes}.
For a $G$-space $X$, we have an $\OG$-space $\Phi X$, defined by, 
\begin{myeq}\label{eqphi}
 \Phi X(G/H):=X^H=\{x\in X|~gx=x~\mbox{for all}~~g\in H\}
\end{myeq}
for each object $G/H$ of $\OG$ and $\Phi X(\ha)(x)=ax,~x\in X^K,$ for morphism $\ha\colon G/H\rightarrow G/K$ of $\OG$. Thus we have a functor $\Phi\colon G\mbox{-}\Top\rightarrow \OG\mbox{-}\Top$.
For a $G$-space $X$ with $G$-fixed point $v$, we have $\OG$-groups $\upi_n(X),~~n\geq 1$, defined by $\upi_n(X)(G/H)=\pi_n(X^H,v)$ for each subgroup $H\leq G$. 
\begin{defn}\label{functor}
An \emph{equivariant local coefficient system} $\MM$ on a $G$-space $X$ is a local coefficient system $\MM(G/H)\colon \pi X^H\rightarrow \Ab$ on $X^H$ for each $H\leq G,$ such that for each morphism $\ha\colon G/H\rightarrow G/K$, there is a natural transformation $\MM(\ha)\colon \MM(G/K)\rightarrow \Phi X(\ha)^*\MM(G/H)$.
\end{defn}
\begin{defn} Let $\upi$ be an $\OG$-group and $M$ be an abelian $\OG$-group. A $\upi$-\emph{module} structure on $M$ is a natural transformation $\uphi\colon \upi\times M\to M,$ that defines a $\upi(G/H)$-module structure on $M(G/H)$ for each subgroup $H\leq G$. 
\end{defn}
Note that, if $X$ is a $G$-connected pointed $G$-space, then a $\upi_1(X)$-module is same as an equivariant local coefficients on $X$.

Let $X$ be a $G$-space and $\MM$ be an equivariant local coefficient system on it. 
Define $C^n_G(X;\MM)$ to be the group of all arrays $$ f=(f(G/H))\in \bigoplus_{H\leq G}C^n(X^H;\MM(G/H))$$ such that $f(G/H)(a\sigma)=\MM(\ha)(\sigma(0))f(G/K)(\sigma)$, where $a^{-1}Ha\subseteq K$ and $\sigma\colon \Delta^n\rightarrow X^K$ is an singular $n$-simplex in $X^K$. We get a cochain complex $\{C^*_G(X;\MM),\delta\}$ by taking the direct sum of of the coboundary of $C^*(X^H;\MM(G/H))$ for all $H\leq G$.
\begin{defn}\label{localbredon}{\cite{m,mm,ms}}
The $n$-th \emph{Bredon cohomology of a $G$-space $X$ with local coefficients} $\MM$ is defined to be
$$H^n_{G}(X;\MM):= H^n(\{C^*_G(X;\MM),\delta\}).$$
\end{defn}
\noindent If $M$ is an (abelian) $\OG$-group and $n\geq 1$ is an integer, then $K_G(M,n)$ denotes an equivariant Eilenberg-Mac~Lane space i.e., any $G$-connected pointed $G$-space $G$-homotopy equivalent to a $G$-CW complex with $\upi_n(K_G(M,n))=M$ and $\upi_i(K_G(M,n))=0$ for $i\neq n$ (see \cite{elm}). 

\noindent Given a $\upi$ module $(M,\uphi)$, there is a sectioned $G$-fibration 
\begin{myeq}\label{eqfib}
K_G(M,n)\hookrightarrow L_{\upi}(M,n)\xrightarrow{p} K_G(\upi,1),~~n\geq 1,
\end{myeq}
\noindent of $G$-connected pointed $G$-spaces $G$-homotopy equivalent to $G$-CW complexes realizing the given module structure as the associated action of $\upi_1(K_G(\upi,1))=\upi$ on $\upi_n(K_G(M,n))=M$ (See \cite{m}). 

In the following we are going to use the vertical homotopy classes of maps $[X, L_{\upi}(M,n)]_ {K_G(\upi,1)}$ for a space $X$ over $K_G(\upi,1)$. Any sectioned fibration $K_G(M,n) \to L \to K_G(\upi,1)$ with a fixed action of $\upi$ on $M$ yields the same homotopy classes.  We will henceforth call the total space of the fibration $L_{\upi}(M,n)$ and any two are $G$-homotopy equivalent. 

A $\upi$-module $\MM=(M,\uphi)$ gives an equivariant local coefficient system on $K_G(\upi,1)$, which we also denote by $\MM$. Suppose that $\theta\colon X\to K_G(\upi,1)$ is a map of $G$-CW complexes. This gives an equivariant local coefficient system $\theta^*\MM$ on $X$. 
Let $[X,L_{\upi}(M,n)]_{K_G(\upi,1)}$ denote the set of homotopy classes of maps in the over category $G\mbox{-}\CW/K_G(\upi,1)$ from $(X,\theta)$ to $(L_{\upi}(M,n),p)$. Note that $[X,L_{\upi}(M,n)]_{K_G(\upi,1)}$ is non-empty as the $G$-fibration \ref{eqfib} is sectioned.
\begin{thm}{\cite{m}}\label{thmol}
 With notations as above, there is a bijection $$H^n_G(X;\theta^*\MM)\cong   [X,L_{\upi}(M,n)]_{K_G(\upi,1)},~n\geq 1.$$ 
\end{thm}
(Compare \cite{ms})

\end{mysubsection}

\begin{mysubsection}{Suspension isomorphism for Bredon cohomology with local coefficients} 
We extend Theorem \ref{thmol} to the case $n=0$ using the suspension isomorphism. This requires a notion of suspension in the category of objects over $K_G(\upi,1)$. 

A basepoint of an object in this category is a section of the map to $K_G(\upi,1)$ and the addition of a disjoint basepoint involves taking a disjoint union with $K_G(\upi,1)$.  Keeping this in mind, we make the following definition. Fix a base point $\ast\in S^1.$
\begin{defn}\label{susp}
Suppose $X$ is a $G$-space over the $G$-space $K_G(\upi,1)$ with $p:X\rightarrow K_G(\upi,1)$ . Define 
$$S_{K_G(\upi,1)}(X):= (X\times S^1) \sqcup K_G(\upi,1)/(x,\ast)\sim p(x).$$
This a $G$-space over $K_G(\upi,1)$, with $G$-map $$ S(p)\colon S_{K_G(\upi,1)}(X)\rightarrow K_G(\upi,1).$$ The map $S(p)$ on the factor $X\times S^1$ is defined to be the projection onto $X$ composed with $p$ and on $K_G(\upi,1)$ is the identity. In the following section we shall observe (Remark \ref{susprem}) that the above definition is a unpointed version of the corresponding definition for ex-spaces. 

Note that $S_{K_G(\upi,1)}(X)$ can be written as the following pushout: 
\begin{mydiagram}[\label{diagpushout}]
{ X \sqcup (S^1\times K_G(\upi,1)) \ar[rr]^{\iota\sqcup \id} \ar[d]_{p\sqcup \pr_2} &&  (X\times S^1) \sqcup (S^1\times K_G(\upi,1))\ar[d] \\
 K_G(\upi,1) \ar[rr] && S_{K_G(\upi,1)}(X)}
\end{mydiagram}

\noindent where $\iota\colon X\hookrightarrow X\times S^1$ is the inclusion $\iota(x)=(\iota,\ast)$ and $\pr_2\colon S^1\times K_G(\upi,1)\to  K_G(\upi,1)$ is the projection on the second factor.
\end{defn}

\begin{prop}\label{susiso}
Let $\MM$ be an equivariant local coefficient system on $K_G(\upi,1)$, given by a $\upi$-module $(M,\uphi)$.
Then
$$H^n_G(S_{K_G(\upi,1)}(X);S(p)^*\MM)\cong H^{n-1}_G(X;p^*\MM)\oplus H^n_G(K_G(\upi,1);\MM).$$
\end{prop}

\begin{proof}
We use the Mayer Vietoris sequence which asserts : For a $G$-CW complex $X$ with $p:X\rightarrow K_G(\upi,1)$ and $X=A\cup B$ with $A$ and $B$ $G$-subcomplexes, there is a long exact sequence
\begin{equation*}
\begin{split}
\cdots\rightarrow H^n_G(X;p^*\MM)\rightarrow  H^n_G(A;p^*\MM)\oplus H^n_G(B;p^*\MM) \rightarrow H^n_G(A\cap B;p^*\MM)\\ \rightarrow H^{n+1}_G(X;p^*\MM)\rightarrow\cdots
\end{split}
\end{equation*}
The existence of this sequence follows from the methods of \cite{mm}. 

Write $S^1=U\cup V$ with $U,V$ open, $*\in U\cap V$ and $U\cap V \simeq S^0$. This induces $S_{K_G(\upi,1)}(X)= S^U(X)\cup S^V(X)$ where 
$$S^U(X)= (X\times U) \sqcup K_G(\upi,1)/(x,\ast)\sim p(x),$$
$$S^V(X)=(X\times V) \sqcup K_G(\upi,1)/(x,\ast)\sim p(x).$$
Then both $S^U(X)$ and $S^V(X)$ are equivariantly homotopic to the mapping cylinder of $p$ and therefore deformation retract to $K_G(\upi,1)$. The space $S^U(X)\cap S^V(X)$ is equivariantly homotopic to the disjoint union $X\sqcup K_G(\upi,1)$. The map 
$$  H^n_G(S^U(X);S(p)^*\MM)\oplus H^n_G(S^V(X);S(p)^*\MM)\rightarrow H^n_G(S^U(X)\cap S^V(X);S(p)^*\MM)$$
is the same as 
$$  H^n_G(K_G(\upi,1);\MM)\oplus H^n_G(K_G(\upi,1);\MM)\stackrel{(p^*,\id) \oplus (p^*,\id)}{\longrightarrow} H^n_G(X;p^*\MM)\oplus H^n_G(K_G(\upi,1);\MM).$$
Hence the result follows from the Mayer Vietoris sequence.
\end{proof}

The adjoint of the suspension is a loop functor which we define as follows. 

\begin{defn}\label{fibreloop1}
Suppose $Y$ is a $G$-space and $p:Y\rightarrow K_G(\upi,1)$ with a section $s:K_G(\upi,1)\rightarrow Y$ so that $p\circ s = \id$. Define 
$$\Omega_{K_G(\upi,1)}(Y):= \{ f\in \Maps(S^1, Y)|~ \exists~ b \in K_G(\upi,1) ~\mbox{so that}~ p(f(t))=b,~f(*)=s(b)\},$$
where $*$ is the basepoint of $S^1$. The map $\Omega(p):\Omega_{K_G(\upi,1)}(Y) \rightarrow K_G(\upi,1)$ is given by  $\Omega(p)(f)= p(f(*))$ and $\Omega(s)\colon K_G(\upi,1) \rightarrow \Omega_{K_G(\upi,1)}(Y) $ is given by $\Omega(s)(b)=\mbox{constant at}~s(b)$.  
\end{defn}

\begin{remark}\label{looprep}
Note that if $p: Y\rightarrow K_G(\upi,1)$ is a Serre fibration with fibre $F$, then $\Omega(p) : \Omega_{K_G(\upi,1)}(Y)\rightarrow K_G(\upi,1)$ is a Serre fibration with fibre $\Omega F$. We apply to the fibration \ref{eqfib}. We obtain 
$$\Omega_{K_G(\upi,1)}(L_{\upi}(M,n))\simeq_G L_{\upi}(M,n-1)$$
as sectioned spaces over $K_G(\upi,1)$.   
\end{remark}

For the next proposition, recall that the mapping space $\Maps_{K_G(\upi,1)}(X,Y)$ of objects in the over category is the pullback $*\rightarrow \Maps(X,K_G(\upi,1))\leftarrow \Maps(X,Y)$. As a set this equals to
$\{ f:X\rightarrow Y | ~ p_Y\circ f = p_X\}.$

\begin{prop}\label{adj}
Suppose $X$ is a $G$-space with $p_X:X\rightarrow K_G(\upi,1)$ and $Y$ is a sectioned $G$-space with $p_Y: Y\rightarrow K_G(\upi,1)$ and $s_Y:  K_G(\upi,1) \rightarrow Y$. Then
$$\Maps_{K_G(\upi,1)}(S_{K_G(\upi,1)}(X), Y) \cong \Maps_{K_G(\upi,1)}( K_G(\upi,1),Y) \times \Maps_{K_G(\upi,1)}(X,\Omega_{K_G(\upi,1)}(Y))$$
\end{prop} 

\begin{proof}
Applying the functor $\Maps_{K_G(\upi,1)}(-,Y)$ to the pushout diagram \ref{diagpushout}, we have a pullback square
$$\xymatrix{\Maps_{K_G(\upi,1)}(S_{K_G(\upi,1)}(X), Y) \ar[r] \ar[d] & \Maps_{K_G(\upi,1)}( (X\times S^1) \sqcup (S^1\times K_G(\upi,1)),Y)  \ar[d] \\ 
\Maps_{K_G(\upi,1)}(K_G(\upi,1),Y)  \ar[r] &  \Maps_{K_G(\upi,1)}(X \sqcup (S^1\times K_G(\upi,1)), Y)}$$
The disjoint unions yield products on the mapping spaces 
\begin{equation*}
\begin{split}
& \Maps_{K_G(\upi,1)}( (X\times S^1) \sqcup (S^1\times K_G(\upi,1)),Y) \\
& \cong \Maps_{K_G(\upi,1)}( X\times S^1,Y)\times \Maps_{K_G(\upi,1)}(  S^1\times K_G(\upi,1),Y),\\
& \mbox{and}\\
& \Maps_{K_G(\upi,1)}(X \sqcup (S^1\times K_G(\upi,1)), Y) \\
& \cong  \Maps_{K_G(\upi,1)}(X, Y)\times \Maps_{K_G(\upi,1)}( S^1\times K_G(\upi,1), Y).
\end{split}
\end{equation*}

The image of $s_Y$ under the map
$$\Maps_{K_G(\upi,1)}(K_G(\upi,1),Y)\rightarrow \Maps_{K_G(\upi,1)}(X, Y)$$
is the map $s_Y\circ p_X$.  It follows that $\Maps_{K_G(\upi,1)}(S_{K_G(\upi,1)}(X), Y)$ is the product of $\Fibre(\Maps_{K_G(\upi,1)}( X\times S^1,Y) \rightarrow \Maps_{K_G(\upi,1)}(X,Y))$ taken over the map $s_Y\circ p_X$ and the pullback $Z$ 
$$\xymatrix{Z \ar[r] \ar[d] & \Maps_{K_G(\upi,1)}( S^1\times K_G(\upi,1),Y)  \ar[d] \\ 
\Maps_{K_G(\upi,1)}(K_G(\upi,1),Y)  \ar[r] &  \Maps_{K_G(\upi,1)}(S^1\times K_G(\upi,1), Y)}$$
Hence  $Z$  equals $\Maps_{K_G(\upi,1)}( K_G(\upi,1),Y) $. Note that $$\Fibre (\Maps_{K_G(\upi,1)}( X\times S^1,Y) \rightarrow \Maps_{K_G(\upi,1)}(X,Y))$$ is equal to 
$$\{ f:S^1\times X \rightarrow Y | ~ f(\ast,x)=s_Y(p_X(x)),~p_Y(f(t,x))=p_X(x)\}.$$
These are precisely the elements of $\Maps_{K_G(\upi,1)}(X,\Omega_{K_G(\upi,1)}(Y))$. Therefore we get 
$$\Maps_{K_G(\upi,1)}(S_{K_G(\upi,1)}(X), Y)\cong \Maps_{K_G(\upi,1)}( K_G(\upi,1),Y) \times \Maps_{K_G(\upi,1)}(X,\Omega_{K_G(\upi,1)}(Y))$$
This completes the proof.
\end{proof}

From Proposition \ref{adj} and Remark \ref{looprep} we extend Theorem \ref{thmol} to the case $n=0$.

\begin{corollary}\label{thmol0}
With notations as in Theorem \ref{thmol}, we have
$$H^0_G(X;p^*\MM)\cong   [X,L_{\upi}(M,0)]_{K_G(\upi,1)}.$$
\end{corollary}

\begin{proof}
From Theorem \ref{thmol} we have 
$$H^1_G(X;p^*\MM)\cong   [X,L_{\upi}(M,1)]_{K_G(\upi,1)}.$$
We use this for $S_{K_G(\upi,1)}(X)$ to get 
$$H^1_G(S_{K_G(\upi,1)}(X);S(p)^*\MM)\cong   [S_{K_G(\upi,1)}(X),L_{\upi}(M,1)]_{K_G(\upi,1)}.$$
From Proposition \ref{susiso} the LHS is 
$$H^0_G(X;p^*\MM)\oplus H^1_G(K_G(\upi,1);\MM).$$
By Proposition \ref{adj} the RHS is 
$$[X,\Omega_{K_G(\upi,1)}(L_{\upi}(M,1))]_{K_G(\upi,1)}\oplus [K_G(\upi,1),L_{\upi}(M,1))]_ {K_G(\upi,1)}.$$
Using Remark \ref{looprep} this simplifies to 
$$[X,L_{\upi}(M,0)]_{K_G(\upi,1)}\oplus [K_G(\upi,1),L_{\upi}(M,1))]_ {K_G(\upi,1)}$$
which is isomorphic to $ [X,L_{\upi}(M,0)]_{K_G(\upi,1)} \oplus H^1_G(K_G(\upi,1);\MM).$ Note that in Proposition \ref{susiso} the summand $H^1_G(K_G(\upi,1);\MM)$ is isomorphic to cohomology of the inclusion $K_G(\upi,1)\to S_{K_G(\upi,1)}X$. The second summand in Proposition \ref{adj} also restricts to an isomorphism on this factor. Therefore the isomorphism 
$$H^0_G(X;p^*\MM)\oplus H^1_G(K_G(\upi,1);\MM) \cong    [X,L_{\upi}(M,0)]_{K_G(\upi,1)}\oplus [K_G(\upi,1),L_{\upi}(M,1))]_ {K_G(\upi,1)}$$
restricts to an isomorphism of the second factors. It follows that 
$$[X,L_{\upi}(M,0)]_{K_G(\upi,1)}\cong H^0_G(X;p^*\MM).$$
\end{proof}
\end{mysubsection}

\begin{mysubsection}{Representation using equivariant crossed complexes}
 If $X$ is a $G$-CW complex then fixed point sets $X^H,~H\leq G,$ are CW complexes (see \cite{dieck}). We have an $\OG$-crossed complex  $$\Pi_G(X)(G/H):=\Pi(X^H).$$
 
 For $T\in \OG$-$\Crs$, define a functorial $G$-space $$\BB^G T:= \Psi \BB T.$$ This is called the \emph{equivariant classifying space} of $T$. Here $\Psi\colon \OG\mbox{-}\Top\to G\mbox{-}\Top$ is the `coalescence functor' of Elmendorf \cite{elm}. It is right adjoint to $\Phi$ (cf. \ref{eqphi}).  
 Using the simplicial enrichment of crossed complexes, the following result is proved in \cite[Theorem 6.2]{rbrown} \cite[Theorem 4.1]{rbrown2}.
 \begin{thm}\label{thogcrsadjoint}
  For a $G$-CW complex $X$ and $T\in\OG\mbox{-}\Crs$, there is a natural homotopy equivalence $$\Maps_G(X,\BB^G T)\simeq \Coh\underline{\Crs}(\Pi_G(X),T)$$
  of Kan complexes, where $\Coh\underline{\Crs}(\Pi_G(X),T)$ denote the simplicial set of homotopy coherent transformations from the diagram $\Pi_G(X)$ to $T$ and left side is the simplicial mapping space in the category $ G\mbox{-}\CW/K_G(\upi,1)$.
 \end{thm}
\noindent See \cite[Definition 3.1]{cp} for homotopy coherent transformations.

Let $\upi$ be an $\OG$-group and $\MM=(M,\uphi)$ be an $\upi$-module. We have the following definition.
\begin{defn}
The $\OG$-crossed complexes $\chi_G(\upi,n)$ and $\uc{n}$, defined by $$\chi_G(\upi,n)(G/H):=\chi(\upi(G/H),n);~~\uc{n}(G/H):=\uch{n},$$ for each object $G/H$ of $\OG$ and for a morphism $\ha\colon G/H\rightarrow G/K$, the maps $$M(\ha)_\ast\colon \chi(\upi(G/K),n)\rightarrow \chi(\upi(G/H),n),~\upi(\ha)_\ast\colon \chi(\upi(G/K),1)\rightarrow \chi(\upi(G/H),1)$$ are naturally induced. Also, we have a map in $\OG$-$\Crs$ 
\begin{myeq}\label{eqcrsfib}
 p\colon \uc{n}\to \chi_G(\upi,n),
\end{myeq}
where $p(G/H)\colon \uch{n}\to \chi(\upi(G/H),n)$ is as defined in Section \ref{ucs}.
\end{defn}
Note that 
$$\BB^G \uc{n} \simeq L_{\upi}(M,n),~~ \BB^G \chi_G(\upi,n)\simeq K_G(\upi, n).$$
\noindent 

\begin{prop}\label{padjoint}
 Applying the functor $\BB^G$ to Eq.\ref{eqcrsfib} yields the classifying $G$-fibration, given in Equation \ref{eqfib}, of Bredon cohomology with local coefficients.
\end{prop}
We fix some notations to be used in the proof. For a functor $U\colon \OGop\to \Ss$, let $U_k$ be composition of $U$ with the functor $\Ss\rightarrow \Set$ that takes a simplicial set to the set of its $k$-simplices. 
  Recall that the two sided bar construction for functors $U\colon\OGop\rightarrow \Ss$ and $V\colon\OG\rightarrow G\mbox{-}\Set$ gives a  bisimplicial set  $(k,l)\mapsto B_l(U_k,\OG,V)$, denoted by $B_\bullet(U_\bullet,\OG,V)$ (using the notation of \cite{ala}). The following lemma follows easily from explicit description of bar construction and definition of a Kan fibration \cite{may}.

\begin{lemma}\label{lem1}
  Suppose $U,U'\colon\OGop\rightarrow \Ss$ and $V\colon\OG\rightarrow G\mbox{-}\Set$ are functors. If $U\xrightarrow{\mu} U'$ is a natural transformation which is Kan fibration for each $G/H\in \OG$, then the induced map on the diagonal simplicial set $\diag B_\bullet(U_\bullet,\OG,V) \xrightarrow{\mu_*^{V}} \diag B_\bullet(U'_\bullet,\OG,V)$ is also a Kan fibration.
 \end{lemma}

\begin{proof}[Proof of Proposition \ref{padjoint}]

 The only thing we need to show: $\BB^G(p)\colon \BB^G\uc{n} \to \BB^G\chi_G(\upi,1)$ is a $G$-fibration, i.e., $\BB^G(p)^H\colon \BB^G\uc{n}^H \to \BB^G\chi_G(\upi,1)^H$ is a Serre fibration for each $H\leq G$. We prove this by showing that $\BB^G(p)^H$ is a geometric realization of a Kan fibration. In this regard we use an explicit model of the Elmendorf functor $\Psi$. 
 Recall that $\BB^G(T)=\Psi(\realz{-}\circ N^{\Delta}\circ T)$ (cf. Definition \ref{dnerve}). 

 Define $\Psi^\Delta (U):=\diag B_\bullet(U_\bullet,\OG,\iota)\in \Ss$, where $\Gamma$ is the category of $G$-$\Set$ and $\iota$ is the inclusion functor $\OG \hookrightarrow G\mbox{-}\Set$. Actually, $\Psi^\Delta (U)$ is a $G$-simplicial set and its $H$-fixed point simplicial set is $\Psi^\Delta (U)^H=\diag B_\bullet(U_\bullet,\OG,\iota^H)$ where the functor $\iota^H\colon \OG\to \Set(\hookrightarrow G\mbox{-}\Set)$ is $G/K\mapsto (G/K)^H=\Hom_{\OG}(G/H,G/K)$ for each $K\leq G$.

 Then we can write $$\BB^G(T)=\realz{\Psi^\Delta (N^{\Delta}\circ T)}.$$ Then $\BB^G(p)=\realz{N^{\Delta}(p)_*^{\iota}}$ and $\BB^G(p)^H=\realz{N^{\Delta}(p)_*^{\iota^H}}$. But $N^{\Delta}(p)(G/H)$ is Kan fibration for each $G/H\in \OG$.  
 Hence the proof follows from lemma \ref{lem1}.
 \end{proof}
 We will use the model $\BB^G \chi_G(\upi,n)$ for $K_G(\upi,n)$. For a $G$-CW complex $(X,\theta)$ over $K_G(\upi,1),$ applying the functor $\Pi_G$ we get an object $\Pi_G(\theta)\colon \Pi_G(X)\to \chi_G(\upi,1)$ of the over category $\OG$-$\Crs/\chi_G(\upi,1)$.
The simplicial mapping space $\Coh\underline{\Crs}((S,\epsilon),(T,\eta))_{\chi_G(\upi,1)}$ in  $\OG\mbox{-}\Crs/\chi_G(\upi,1)$ is defined by the pull-back square: 
\begin{mydiagram}[\label{diagcoh}]
{\Coh\underline{\Crs}((S,\epsilon),(T,\eta))_{\chi_G(\upi,1)}\ar[rr]\ar[d] && \Coh\underline{\Crs}(S,T)\ar[d]^{\eta_*}\\
\ast \ar[rr] && \Coh\underline{\Crs}(S,\chi_G(\upi,1))
}
\end{mydiagram}
In view of Theorem \ref{thogcrsadjoint}, it follows from general category theory that 
\begin{myeq}\label{eqadjoint}
 \Maps_G(X,\BB^G T)_{K_G(\upi,1)}\simeq \Coh\underline{\Crs}(\Pi_G(X),T)_{\chi_G(\upi,1)}.
\end{myeq}
\noindent

 We obtain the following representation of Bredon cohomology with local coefficients in $\OG$-diagram of crossed complexes.
\begin{thm}\label{thmfirst}
Let $\upi$ be an $\OG$-group and $\MM=(M,\uphi)$ an $\upi$-module. For $(X,\theta)\in G\mbox{-}CW/K_G(\upi,1)$ we have 
 $$H^n_G(X;\theta^*\MM)\cong  \pi_0 \Coh\underline{\Crs}(\Pi_G(X),\uc{n})_{\chi_G(\upi,1)}.$$ 
\end{thm}
\begin{proof}
 We specialize $T~\in \OG\mbox{-}\Crs/\chi_G(\upi,1)$ to be $ \uc{n}\xrightarrow{p} \chi_G(\upi,n)$ in Eq. \ref{eqadjoint}. In view of Proposition \ref{padjoint}, $\pi_0 \Maps_G(X,\BB^G T)_{K_G(\upi,1)}$ is the right hand side of the isomorphism in Theorem \ref{thmol} for $n\geq 1$ and Corollary \ref{thmol0} for $n=0$. Hence applying $\pi_0$ to Eq \ref{eqadjoint} gives the desired result.
\end{proof}

If $G$ is trivial, then the simplicial set of homotopy coherent transformations between $\OG$-diagrams in $\Crs$ reduces to the ordinary simplicial enrichment of $\Crs$. Further, $\upi$ reduces to a group $\pi$ and $(M,\uphi)$ is a $\pi$-module $(A,\phi)$. This gives local coefficients $\A$ on $K(\pi,1)$. Let $X$ be a CW complex and $\theta\colon X\to K(\pi,1)$ be a map of CW complexes. 
\begin{corollary}\label{trloc}
 With notations as above, $$H^n(X;\theta^*\A)\cong [\Pi(X),\chi_{\phi}(A,n))]_{\chi(\pi,1)}$$ where right hand side denotes the homotopy classes of maps in $\Crs/\chi(\pi,1)$ from  $ \Pi(X)\xrightarrow{\Pi(\theta)} \chi(\pi,1)$ to $\chi_{\phi}(A,n)\xrightarrow{p}\chi(\pi,1)$.
\end{corollary}
When $X$ is a reduced CW complex, Corollary \ref{trloc} agrees with Theorem \ref{thbprop} for $n\geq 2$ by taking $\alpha=\Pi(\theta)$ and noting that homotopy classes of maps in $\Crs/\chi(\pi,1)$ from  $ \Pi(X)\xrightarrow{\Pi(\theta)} \chi(\pi,1)$ to $\chi_{\phi}(A,n)\xrightarrow{p}\chi(\pi,1)$ is same as homotopy classes of maps $\Pi(X)\to \chi_{\phi}(A,n)$ inducing $\alpha$ in fundamental groups.
\end{mysubsection}

\section{Representability of cohomology with local coefficients as a parametrized spectrum}\label{slocparspc}
 In this section we use Corollary \ref{trloc} to represent the cohomology with local coefficients using parametrized spectra. Following \cite{par} we construct an $\Omega$-prespectrum $J_\pi A$ over the parameter space $K(\pi,1)$ associated to a $\pi$-module $A$, and prove that the cohomology theory on the category of spaces over $K(\pi,1)$ associated to this spectrum is the cohomology with local coefficients. We use the classifying space functor $\BB$ from crossed complexes to spaces to define this parametrized spectrum. Recall the definition of parametrized spectra following \cite{par}.

We work in the category of compactly generated weak Hausdorff spaces.
\begin{defn}\label{dexspace}
An \emph{ex-space} over $B$ is a triple $(X,p,s)$ where $X$ is a topological space and $B\stackrel{s}{\rightarrow} X\stackrel{p}{\rightarrow} B$ such that $p\circ s = id$.

A morphism of ex-spaces $(X,p,s)\stackrel{f}{\rightarrow} (Y,p',s')$ is a map $f\colon X\rightarrow Y$ such that $p'\circ f=p$ and $f\circ s=s'$.
\end{defn}
 We denote category of all ex-spaces over the space $B$ by $\TT_B$. The category $\TT_B$ is enriched over topological spaces and has all colimits and limits. There is a fibrewise smash product $\wedge_B$ and a fibrewise mapping space $F_B$ of ex-spaces, and these are adjoint.
We recall the relevant definitions briefly (see \cite{par} for more details).
\begin{defn}
 Let $X\xrightarrow{p}B$ and $Y\xrightarrow{p'}B$ be spaces over $B$.
\begin{itemize}
 \item The pullback of $X\xrightarrow{p} B \xleftarrow{p'} Y$ is denoted by $X\times _B Y$. It has an evident map to $B$.
\item The pullback of $\ast\rightarrow \Maps(X,B)\xleftarrow{p'^*} \Maps(X,Y)$ is denoted by $\Maps_B(X,Y)$.
\end{itemize}
\end{defn}
\begin{defn}
 Let $(X,p,s), (Y,p',s')$ be ex-spaces over $B$.
\begin{itemize}
 \item The pushout of $X\xleftarrow{s} B\xrightarrow{s'} Y$ is denoted by $X\vee_B Y$, the fibrewise wedge of $X$ and $Y$. It has an naturally induced map $X\vee_B Y\rightarrow X\times_B Y$ over $B$.
\item The pushout of $\ast_B\leftarrow X\vee_B Y \rightarrow X\times_B Y$ is the smash product $X\wedge_B Y$ in $\TT_B$.
\end{itemize}
\end{defn}

The definition of the mapping space in $\TT_B$ is more complicated. We refer to \cite{par} for details.  For objects $X,Y \in \TT_B$ the mapping space is written as $F_B(X,Y)$ and there is  a correspondence 
$$F_B(X\wedge_B Y,Z)\cong F_B(X,F_B(Y,Z)).$$

For a pointed space $Y$, let $Y_B$ denote the ex-space $(Y\times B,p,s)$, where $p$ is the projection onto the second factor and $s$ is the cross section determined by the base point of $Y$. The \emph{fibrewise loop functor} is then defined as $\Omega_B(X):=F_B(S^1_B,X)$. In, general, for a finite dimensional inner product space $V$, let $S^V$ denote the one point compactification of $V$. So we have $$\Sigma_B^V(X):=X\wedge_B S^V_B,~~\Omega_B(X):=F_B(S_B^V,X).$$
These definitions match up with the corresponding ones in Section 3. 
\begin{remark}\label{susprem}
Note that one can recover Definition \ref{susp} from the above. For a space $X$ over $B$ one can construct an ex-space over $B$ by adding a ``disjoint basepoint". We denote $X_{+B}=X\sqcup B$ with the evident projection to $B$ and section from $B$. The above formula readily yields 
$$\Sigma_B(X_{+B}) \simeq S_B(X)$$ 

\end{remark}

\begin{remark}\label{fibreloop}
For an ex-space $Y$ over $B$ the fibrewise loops $\Omega_B Y$ matches the definition \ref{fibreloop1}. That is,
$$\Omega_BY\cong \{ f\in \Maps(S^1,Y)|~\exists ~b \in B ~\mbox{so ~that}~ p(f(t))=b,~f(*)=s(b)\}.$$
 Note that this can be defined by the following pullbacks
$$\xymatrix{ M_B(S^1,Y) \ar[r] \ar[d] & \Maps(S^1,Y) \ar[d] \\
                   B \ar^{e}[r] & \Maps(S^1,B) }$$
where $e(b)$ is the constant map at $b$. Then one has a pullback 
$$\xymatrix{ \Omega_B Y \ar[r] \ar[d] & M_B(S^1,Y) \ar^{ev_*}[d] \\
                   B \ar^{s_Y}[r] & Y }$$
 
\end{remark} 

A map $(X,p,s)\stackrel{f}{\rightarrow} (Y,p',s')$ of ex-spaces over $B$ is called an \emph{$q$-equivalence} if $X\xrightarrow{f}Y$ is a weak equivalence of spaces (forgetting the ex-space structure). The notion of \emph{$qf$-fibrations} is defined as maps satisfying the right lifting property with respect to certain inclusions $S^n_+ \rightarrow D^{n+1}$ where $S^n_+$ is the closed upper hemisphere of $D^{n+1}$ (see \cite[Definition $6.2.3$]{par} for details). Here `$f$' refers to fibrewise. Since the inclusion of the upper hemisphere is an acyclic cofibration the condition is satisfied for Serre fibrations. Thus, if the map $X\rightarrow B$ is a Serre fibration then $X$ is $qf$-fibrant. 

\begin{thm}[Theorem 6.2.6 \cite{par}]
 The category $\TT_B$ of ex-spaces over $B$ is a well-grounded model
category with respect to the $q$-equivalences, $qf$-fibrations, and $qf$-cofibrations.
\end{thm}

To define spectra, we first fix a countable infinite dimensional inner product space $\UU$. Following \cite{par}, we define a prespectrum $E$ over $B$. 
\begin{remark}
A prespectrum/spectrum over $B$ always stands for an object in the category of parametrized spectra over a space $B$. It should not be confused as an object in the over category of some fixed prespectrum/spectrum $B$.
\end{remark}
\begin{defn}\label{dparasec}
Let $B$ be a fixed base space.
\begin{itemize}
 \item A \emph{prespectrum} $E$ over $B$ consists of ex-spaces $E(V)$ over $B$ for each finite dimensional subspace $V$ of $\UU$, together with maps of ex-spaces, (called structure maps) $\sigma^{V,W}\colon \Sigma^{W-V}_B E(V)\rightarrow E(W)$ for $V\subseteq W$, such that the following are satisfied:
\begin{enumerate}[(i)]
 \item For each finite dimensional $V\subset \UU$, $\sigma^{V,V}=id$.
\item For finite dimensional $V,W,Z\subset \UU$, $\sigma^{V,Z}=\sigma^{W,Z}\circ\Sigma_B^{Z-W}\sigma^{V,W}.$
\end{enumerate}
\item A prespectrum $E$ over $B$ is level $qf$-fibrant if each $E(V)$ is a $qf$-fibrant ex-space over $B$.
\item An $\Omega$-prespectrum $E$ over $B$ is defined to be a level $qf$-fibrant prespectrum over $B$ whose adjoint structure maps $\widetilde{\sigma}^{V,W}\colon E(V) \rightarrow \Omega_B^{W-V} E(W)$ are $q$-equivalences of ex-spaces over $B$.
\end{itemize}
\end{defn}
Consider the crossed complex $\chi_{\phi}(A,n)$ associated to a $\pi$-module $(A,\phi)$ (cf. Section \ref{ucs}). Since the classifying space of a crossed complex is the geometric realization of a simplicial set, $\BB(\chi_{\phi}(A,n))$ is compactly generated and weak Hausdorff.
\begin{prop}\label{pexsp}
The classifying space $\BB(\chi_{\phi}(A,n))$ is an ex-space over $K(\pi,1)= \BB(\chi(\pi,1))$ which is $qf$-fibrant.
\end{prop}
\begin{proof}
The map $p$ is obtained by applying $\BB$ to the map $\chi_{\phi}(A,n)\rightarrow \chi(\pi,1)$. The map $\chi(\pi,1) \rightarrow \chi_{\phi}(A,n)$ for $n\geq 2$ is the identity at level 1 and is the inclusion of the identity at higher levels. For $n=1$ this is the inclusion of $\pi$ in $\pi \ltimes A$ at the first level. For $n=0$ the map takes $\pi$ to the endomorphisms of the identity in $A$ in the groupoid $\mathit{Gpd}(\pi,A)$. It is clear that the composition $\chi(\pi,1) \rightarrow \chi_{\phi}(A,n) \rightarrow \chi(\pi,1)$ is the identity.

To complete the proof we need to show that $\BB(\chi_{\phi}(A,n))\rightarrow \BB(\chi(\pi,1))$ is a Serre fibration, which in view of the discussion above implies that $\BB(\chi_{\phi}(A,n))$ is a $qf$-fibrant ex-space. Since the geometric realization of a Kan fibration is a Serre fibration, it is enough to show that $N^{\Delta}(\chi_{\phi}(A,n))\rightarrow N^{\Delta}(\chi(\pi,1))$ is a Kan fibration of simplicial sets, which by (\cite[Proposition $6.2$]{rbrown3}) is equivalent to the condition that $\chi_{\phi}(A,n) \rightarrow \chi(\pi,1)$ is a fibration of crossed complexes.

Recall that a map $p: G \to H$ of groupoids is called a fibration if for each object $x$ of $G$ and each morphism $b$ of $H$ starting at $p(x)$ there is a morphism $e$ of $G$ starting at $x$ such that $p(e)=b$. A map between crossed complexes is a fibration if it is a fibration of groupoids on level $1$ and surjective on higher levels (see \cite{rbrown4, rbrown1}). This is satisfied by $\chi_{\phi}(A,n)\rightarrow \chi(\pi,1)$. This completes the proof of the proposition.

\end{proof}

We will prove that $$J_\pi A(\R^n):= \BB(\chi_{\phi}(A,n))$$ defines an $\Omega$-prespectrum over $\BB(\chi(\pi,1))$. This defines a spectrum indexed over a cofinal collection of inner product subspaces of $\UU$, we can extend it to the entire collection by the formula
$$J_\pi A(V)=\Omega^{\R^n-V}_{K(\pi,1)}J_\pi A(\R^n),$$
 where $n$ is the minimum positive integer for which $V\subset\R^n$. By the above, $J_\pi A$ is level $qf$-fibrant on the indexing spaces $\R^n$. It follows that $J_\pi A(V)\rightarrow K(\pi,1)$ is a $qf$-fibration since this is equal to $\Omega^{\R^n-V}_{K(\pi,1)}J_\pi A(\R^n)$ which is written as an iterated pullback, and that fibrations are preserved under pullback.

It remains to construct the structure maps of $J_\pi A(V)\rightarrow \Omega_{K(\pi,1)}^{W-V}J_\pi A(W)$ and show that they are weak equivalences, and again it suffices to prove this for the cofinal indexing collection $\lbrace \R^n\rbrace$. In this regard, it remains to construct maps $\BB(\chi_{\phi}(A,n))\rightarrow \Omega_{\BB(\chi(\pi,1))} \BB(\chi_{\phi}(A,n+1))$ and to show that these are weak equivalences.

We define a construction similar to $F_B(S^1_B,-)$ in the category of crossed complexes.  The classifying space of this construction is weakly equivalent to $F_B(S^1_B,-)$ as ex-spaces over the space $B$. Fix the model $\chi(\ints,1)=\Pi(S^1)$ of $S^1$ in crossed complexes (defined using the CW complex structure of $S^1$ with one $1$-cell and one $0$-cell).

We can form a map $\chi(\pi,1)\stackrel{\iota}{\rightarrow} \CRS(\chi(\ints,1),\chi(\pi,1))$ which maps each point to the corresponding ``constant map"; we then define $\CRS_{\chi(\pi,1)}(\chi(\ints,1),\chi_{\phi}(A,n))$ to be the pullback
$$\xymatrix{  \CRS_{\chi(\pi,1)}(\chi(\ints,1),\chi_{\phi}(A,n))\ar[r] \ar[d] & \CRS(\chi(\ints,1),\chi_{\phi}(A,n)) \ar[d]\\
              \chi(\pi,1) \ar[r]^{\iota}              & \CRS(\chi(\ints,1),\chi(\pi,1)) }$$
Then we further form the pullback to define $F^{\Crs}_{\chi(\pi,1)}( \chi(\ints,1),\chi_{\phi}(A,n))$
$$\xymatrix{F^{\Crs}_{\chi(\pi,1)}(\chi(\ints,1),\chi_{\phi}(A,n))\ar[r] \ar[d] & \CRS_{\chi(\pi,1)}(\chi(\ints,1),\chi_{\phi}(A,n)) \ar[d]^{\epsilon}\\
              \chi(\pi,1) \ar[r]              & \chi_{\phi}(A,n) }$$
where $\epsilon$ is the evaluation at the base point map. Note that the pullback in the category of crossed complexes is obtained by taking the pullback levelwise, i.e, the pullback $P$ of $\CC\rightarrow \DD \leftarrow \CC'$ is described by $P_n=C_n\times_{D_n} C'_n$.

We calculate $\CRS(\chi(\ints,1),\chi(\pi,1))$ using its definition in terms of $m$-fold homotopies (See Definition \ref{mhomotopy}). We have
\begin{equation*}
 \begin{split}
  & \CRS_0(\chi(\ints,1),\chi(\pi,1))= \pi,\\
  & \CRS_1(\chi(\ints,1),\chi(\pi,1))(f,g)=\{ x\in \pi| f=xg x^{-1}\},\\
  & \CRS_m(\chi(\ints,1),\chi(\pi,1))(f,f)=* ~\mathit{if}~ m\geq 2.
 \end{split}
\end{equation*}
Next we calculate $\CRS(\chi(\ints,1),\chi_{\phi}(A,n))$. 
$$\CRS_0(\chi(\ints,1),\chi_{\phi}(A,n))= \left\{ \begin{array}{ll}
 \Hom(\ints,\pi)\cong \pi &\mbox{ if $n\geq 2$} \\
\pi \ltimes A   &\mbox{ if $n=1$}
       \end{array} \right.
       $$

$$\CRS_1(\chi(\ints,1),\chi_{\phi}(A,n))(f,g)= \left\{ \begin{array}{ll}
 \{ x\in \pi| f=xgx^{-1}\} &\mbox{ if $n\geq 3$} \\
\{ x\in \pi | xfx^{-1}=g\} \times A &\mbox{ if $n=2$}\\
\{ x\in \pi \ltimes A | xfx^{-1}=g \}  &\mbox{ if $n=1$}
       \end{array} \right.
       $$

For $n\geq 3,~m\geq 2$,
$$\CRS_m(\chi(\ints,1),\chi_{\phi}(A,n))(f,f)= \left\{ \begin{array}{ll}
 \ast & \mbox{ if $m\neq n-1,n$}\\
  \Hom_\pi(\pi,A)\cong A & \mbox{ if $m=n-1$}\\
     \Maps(*,A)  &\mbox{ if $m=n$}
       \end{array} \right.
       $$
For $n=2,~m\geq 2$ we have
$$\CRS_m(\chi(\ints,1),\chi_{\phi}(A,2))(f,f)= \left\{ \begin{array}{ll}
 \Hom(*,A) & \mbox{ if $m=2$}\\
  \ast & \mbox{ if $m>2$}
           \end{array} \right.
$$
and for $n=1,~m\geq 2$ we have
$$\CRS_m(\chi(\ints,1),\chi_{\phi}(A,1))(f,f)= \ast $$

We have the following proposition.
\begin{prop}\label{pparcrs}
 For a $\pi$-module $(A,\phi)$, $F^{\Crs}_{\chi(\pi,1)}(\chi(\ints,1),\chi_{\phi}(A,n)) \cong \chi_{\phi}(A,n-1)$.
\end{prop}
\begin{proof}
To justify the statement of the proposition, we make explicit calculations; considering separately the cases $n\geq 3,$ $n=2$ and $n=1$.

\noindent \textbf{Case I, n$\geq$ 3}:
 We need to compute the map $\iota \colon  \chi(\pi,1) \rightarrow \CRS(\chi(\ints,1),\chi(\pi,1))$: this corresponds to the map which takes a point of $\chi(\ints,1)$ to the constant map based at a point in $\chi(\pi,1)$. It can be factored as $\chi(\pi,1)\stackrel{\cong}{\rightarrow} \CRS(*,\chi(\pi,1))\rightarrow \CRS(\chi(\ints,1), \chi(\pi,1))$. The crossed complex $*$ is the complex which is trivial at each level. The latter map is induced from the map $\chi(\ints,1)\rightarrow *$.

The map $\iota\colon \chi(\pi,1) \rightarrow \CRS(\chi(\ints,1),\chi(\pi,1))$ on the zero level takes $*$ to $\id \in \pi$, on the 1-level takes $\pi$ isomorphically to $\pi=\CRS_1(\chi(\ints,1),\chi(\pi,1))(*,*)$. So, when we form the pullback we have

$$\CRS_{\chi(\pi,1)}(\chi(\ints,1),\chi_{\phi}(A,n))_m= \left\{ \begin{array}{ll}
 \ast &\mbox{ if $m=0$} \\
 \pi &\mbox{ if $m=1$}\\
  A & \mbox{ if $m=n-1,n$}\\
\ast &\mbox{else}
           \end{array} \right.
$$

To form the second pullback we need to calculate the map
$$\epsilon\colon \CRS_{\chi(\pi,1)}(\chi(\ints,1),\chi_{\phi}(A,n))\rightarrow \chi_{\phi}(A,n),$$
given by evaluation at a base-point. The base-point in $\chi(\ints,1)$ is given by a map $*\rightarrow \chi(\ints,1)$ and the map $\epsilon$ is the composite
$$\CRS_{\chi(\pi,1)}(\chi(\ints,1),\chi_{\phi}(A,n))\rightarrow \CRS_{\chi(\pi,1)}(*,\chi_{\phi}(A,n))\cong \chi_{\phi}(A,n).$$
This map is an isomorphism at the level $1$ and $n$ and the trivial map at the other levels. The map $\chi(\pi,1)\rightarrow \chi_{\phi}(A,n)$ is an isomorphism at level $1$ and trivial at other levels. So when we form the pullback the level $n$ part goes away and we get,

$$F^{\Crs}_{\chi(\pi,1)}(\chi(\ints,1),\chi_{\phi}(A,n))_m= \left\{ \begin{array}{ll}
 \pi &\mbox{ if $m=1$} \\
 A &\mbox{ if $m=n-1$}\\
\ast &\mbox{if $m\neq 1,n-1$}
           \end{array} \right.
                              $$

Thus we get $F^{\Crs}_{\chi(\pi,1)}(\chi(\ints,1),\chi_{\phi}(A,n))\cong \chi_{\phi}(A,n-1)$ for $n\geq 3$.

\noindent \textbf{Case II, n=2:}
 The map $$\CRS(\chi(\ints,1),\chi_{\phi}(A,2))\rightarrow \CRS(\chi(\ints,1),\chi(\pi,1))$$ is the trivial map on level $i$ for $i>1$ and the projection $$\lbrace x\in \pi| xfx^{-1}=g\rbrace \times A \rightarrow \lbrace x\in \pi | xfx^{-1}=g\rbrace $$ for $i=1$. Forming the pullback we get,
$$\CRS_{\chi(\pi,1)}(\chi(\ints,1),\chi_{\phi}(A,2))_m= \left\{ \begin{array}{ll}
 \ast &\mbox{ if $m=0$} \\
 \pi\ltimes A &\mbox{ if $m=1$}\\
  A & \mbox{ if $m=2$}\\
\ast &\mbox{if $m>2$}
           \end{array} \right.
                              $$

The map
$$\epsilon\colon \CRS_{\chi(\pi,1)}(\chi(\ints,1),\chi_{\phi}(A,2))\rightarrow \chi_{\phi}(A,2)$$
given by evaluation at the basepoint can be calculated similarly as before as an isomorphism on level 2 and the projection $\pi\ltimes A \rightarrow \pi$ on level 1, and trivial on other levels. For the second pullback we get,
$$F^{\Crs}_{\chi(\pi,1)}(\chi(\ints,1),\chi_{\phi}(A,2))_1=\pi\ltimes A;~F^{\Crs}_{\chi(\pi,1)}(\chi(\ints,1),\chi_{\phi}(A,2))_i=* \,\mbox{for all}~ i\neq 1.$$
Thus we have proved that $F^{\Crs}_{\chi(\pi,1)}(\chi(\ints,1),\chi_{\phi}(A,2))\cong \chi_{\phi}(A,1)$.

\noindent \textbf{Case III, n=1}:
 We fix the notation $pr$ for the projection $\pi \ltimes A \rightarrow \pi$. The map of mapping spaces $\CRS(\chi(\ints,1),\chi_{\phi}(A,1))\rightarrow \CRS(\chi(\ints,1),\chi(\pi,1))$ is $\pr$ on level 0 and 1 and trivial in higher levels. Forming the pullback we get,
$$\CRS_{\chi(\pi,1)}(\chi(\ints,1),\chi_{\phi}(A,1))_0=A.$$
We use the notation $(R,S)$ for a groupoid with objects $R$ and morphisms $S$. In this notation the next level is the groupoid pullback:
$$(*,\pi) \rightarrow (\pi, \Hom(f,g)=\lbrace x\in \pi| xfx^{-1}=g\rbrace) \stackrel{\pr}{\leftarrow} (\pi\ltimes A,\Hom(f,g)= \lbrace x\in \pi\ltimes A| xfx^{-1}=g\rbrace).$$ 
The groupoid pullback has object set $A$. The set $\Hom(l_1,l_2)$ is just the pullback of
$$\Hom(*,*)\rightarrow \Hom(*,*)\leftarrow  \Hom((*,l_1),(*,l_2)).$$
In the last set we have those elements $x=(a,l)$ (this means $x^{-1}=(a^{-1},-a^{-1}(l))$) which satisfy $x(*,l_1)x^{-1} =(*,l_2)$, i.e. $(*, a^{-1}(l+l_1)-a^{-1}(l))= (*,l_2)$. This reduces to the equation $a^{-1}(l_1)=l_2$. This means we have
$$\CRS_{\chi(\pi,1)}(\chi(\ints,1),\chi_{\phi}(A,1))_1(l_1,l_2)=\lbrace a\in \pi | a^{-1}(l_1)=l_2\rbrace \times A$$
and
$$\CRS_{\chi(\pi,1)}(\chi(\ints,1),\chi_{\phi}(A,2))_i(l)=*  ~~\mbox{  for all}~ i \geq 2.$$

The next step is to calculate the map
$$\epsilon\colon \CRS_{\chi(\pi,1)}(\chi(\ints,1),\chi_{\phi}(A,1))\rightarrow \chi_{\phi}(A,1)$$
which is given by evaluation at the base-point and can be calculated similarly as before. At the $0$-level it maps $A$ to $*$. At the $1$-level on $\Hom(l_1,l_2)$ is the inclusion
$$\lbrace a\in \pi | a^{-1}(l_1)=l_2\rbrace \times A \rightarrow \pi \ltimes A.$$
At other levels this is the trivial map. So again the pullback becomes a pullback of groupoids (just considering the 0 and 1-levels)
$$(*,\pi)\rightarrow (*,\pi\ltimes A)\leftarrow (A,\Hom(l_1,l_2)= \lbrace a\in \pi | a^{-1}(l_1)=l_2\rbrace \times A).$$
The object set of this pullback is $A$. In the pullback groupoid the morphisms between $l_1$ and $l_2$ are
$$\lbrace(a,(b,l)| (a,0)=(b,l)\, ~\mbox{and}~ b^{-1}(l_1)=l_2\rbrace.$$
This forces $a=b$ and $l=0$. Then we get exactly the same morphisms as those of $\mathit{Gpd}(\pi,A)$. Therefore we obtain

$$F^{\Crs}_{\chi(\pi,1)}(\chi(\ints,1),\chi_{\phi}(A,1))_m= \left\{ \begin{array}{ll}
 A &\mbox{ if $m=0$} \\
 \mathit{Gpd}(\pi,A) &\mbox{ if $m=1$}\\
 \ast &\mbox{if $m>1$}
           \end{array} \right.
                              $$

This is precisely the description of $\chi_{\phi}(A,0)$. This completes the second part.
\end{proof}

\begin{prop}\label{pbcrs}
For a $\pi$-module $(A,\phi)$, $$\BB F^{\Crs}_{\chi(\pi,1)}(\chi(\ints,1),\chi_{\phi}(A,n)) \simeq F_{K(\pi,1)}(S^1_{K(\pi,1)}, \BB\chi_{\phi}(A,n)).$$
\end{prop}
\begin{proof}
The map $S^1\rightarrow \BB\Pi(S^1)\simeq K(\ints,1)$ is a weak equivalence. Also we know that $\Maps(X,\BB(\CC))\simeq \BB(\CRS(\Pi(X),\CC))$ (cf. Theorem \ref{thadjoint}). Therefore, $\BB\CRS(\chi(\ints,1),\CC)\simeq \Maps(S^1,\BB\CC)$. Since the nerve functor is a right adjoint and geometric realization takes pullbacks to pullbacks, the functor $\BB$ takes pullbacks to pullbacks. The proof follows.
\end{proof}
Combining the above propositions, we have the following result.

\begin{thm}\label{thnoneq}
The spaces $\{J_\pi A(V)\}_{V\subset \R^\infty}$ form an $\Omega$-prespectrum over $K(\pi,1)=\BB\chi(\pi,1)$.
\end{thm}
\begin{proof}
Combining Proposition \ref{pparcrs} and Proposition \ref{pexsp}, we see that there is a map $$\BB \chi_{\phi}(A,n)\rightarrow \Omega_{K(\pi,1)} \BB\chi_{\phi}(A,n+1)$$ over $K(\pi,1)=\BB\chi(\pi,1)$ which is a $q$-equivalence. Also it follows from Proposition \ref{pexsp}, that the spaces $\BB\chi_{\phi}(A,n)$ are a $qf$-fibrant ex-space over $K(\pi,1).$ We have seen how to extend these results from the cofinal collection $\{ \R^n\}$ to the collection of all inner product subspaces of $\R^\infty$. Therefore, $\{J_\pi A(V)\}_{V\subset \R^\infty}$ is an $\Omega$-prespectrum.
\end{proof}

We recall the definition of cohomology theory defined by a parametrized spectrum $J$ over a base space $B$.
\begin{defn}\cite[Definition 20.2.4]{par}
 Let $J,~E$ be spectra over $B$. For integers $n$, define the $n$-th \emph{$J$-cohomology groups} of $E$ as $$J^n(E):=\pi_{-n}(r_*F_B(E,J)),$$ where $r_*$ is the base change functor with respect to the map $B\to \ast$ (cf. \cite[Section 2.1, 11.4]{par}). 
 For an ex-space $X$ over $B$, taking $E=\Sigma_B^{\infty}X$ defines the $J$-cohomology groups of $X$ determined by the spectrum $J$.
\end{defn}
These cohomology groups can also be written as
$$ J^n(E) = [S_B^{-n}, F_B(E, J)]_B \cong [E,F_B(S_B^{-n}, J)]_B \cong [E,\Sigma_B^n J]_B.$$ Here $S_B^{-n}$ is the $n$-fold desuspension of the sphere spectrum (see \cite[Definition 11.3.5]{par}).

Recall from Remark \ref{susprem} that for a space $X$ over $K(\pi,1)$ the notation $X_{+K(\pi,1)}$ denotes the ex-space $X\sqcup K(\pi,1)$ over $K(\pi,1)$. Note that the category $\Top/K(\pi,1)$ has a terminal object $K(\pi,1)\stackrel{\id}{\rightarrow}K(\pi,1)$ and the based objects of this category are precisely the ex-spaces. It follows that the functor $(-)_{+K(\pi,1)}$ is  a left adjoint to the forgetful functor from the category of ex-spaces to the category of spaces over $K(\pi,1)$. Moreover
$$S_{K(\pi,1)}(X)\cong \Sigma_{K(\pi,1)}(X_{+K(\pi,1)})$$

We are now in a position to show that the cohomology defined by the parametrized spectrum $J_\pi(A)$ is the cohomology with local coefficients.
\begin{thm}
\label{thm:locrepspec}
Let  $(A,\phi)$ be a $\pi$-module and $\theta:X\to K(\pi,1)$ be a space over $K(\pi,1)$. The we have $$H^n(X;\theta^*\A) \cong J_{\pi}A^n(X_{+K(\pi,1)}),$$ where left hand side is the cohomology of $X$ with local coefficients $\theta^*\A$ induced from $(A,\phi)$ by the map $\theta.$
\end{thm}
\begin{proof}
We have,
\begin{align}
J_{\pi}A^n(X_{+K(\pi,1)}) & =  J_\pi A^n(\Sigma^\infty_{K(\pi,1)} X_{+K(\pi,1)})\nonumber\\
 &= [\Sigma^\infty_{K(\pi,1)} X_{+K(\pi,1)},\Sigma^n_{K(\pi,1)} J_\pi A]_{K(\pi,1)}\nonumber \\
 &= [X_{+K(\pi,1)}, \Omega^\infty_{K(\pi,1)} \Sigma^n_{K(\pi,1)} J_\pi A]_{K(\pi,1)} \nonumber \\
 &=  [X_{+K(\pi,1)}, \BB(\chi_{\phi}(A,n)]_{K(\pi,1)}.\nonumber
\end{align}
The latter is the group obtained by taking homotopy classes of the maps (in $\Top/K(\pi,1)$) from $X\xrightarrow{\theta} K(\pi,1)$ to $\BB(\chi_{\phi}(A,n))\xrightarrow{\BB(p)}\BB(\chi(\pi,1))$. If $X$ is a CW complex, via the homotopy adjunction between $\Pi$ and $\BB$ as given in Eq \ref{eqadjoncrs}, this is exactly $[\Pi(X),\chi_{\phi}(A,n)]_{\chi(\pi,1)}$ which was shown to be $H^n(X;\theta^*\A)$ in Corollary \ref{trloc}.
\end{proof}

\section{Representation as a parameterized spectrum in the equivariant case}\label{squipar}

The definition of a parametrized ($\Omega$-)$G$-prespectrum over base $B$ is entirely analogous to  Definition \ref{dparasec} where $B$ is a $G$-space, $\UU$ is a $G$-universe and each $E(V)$ is an ex-$G$-space \cite[Definition 11.2.16, Definition 12.3.6]{par}. In this section we represent Bredon cohomology with local coefficients with a parametrized $G$-prespectrum over the $G$-space $K_G(\upi,1)$ (notations as in section  \ref{seqtheo}).

Let $\EE$ denote the category with two objects $s,t$ and the morphisms generated by $i\colon s\rightarrow t$ and $p\colon t\rightarrow s$ such that $p\circ i$ is the identity. We call the diagram category $\Top^\EE$ the ``ex-category" of spaces, and by an ``ex-functor" we mean a functor $\OGop \rightarrow \Top^\EE$.
We have observed in Section \ref{sbre} that, for an $\OG$-group $\upi$ and an $\upi$-module $\MM=(M,\uphi)$, we can construct $\OG$-crossed complexes $\uc{n}$ over the $\OG$-crossed complex $\chi_G(\upi,1)$. By applying the classifying space functor for each $n$ we get an ex-functor $\EE^\MM_n\colon \OGop \rightarrow \Top^\EE$ for each $n\geq 0$, which associates $G/H$ to the ex-
space $\BB\chi_{\uphi(G/H)}(M(G/H),n)$ over $\BB\chi(\upi(G/H),1)$.

In the previous section we have observed that for a fixed $H\leq G$, the ex-spaces $\{\EE^\MM_n(G/H)\}_{n\geq 0}$ form a parametrized $\Omega$-prespectrum, denoted by $\EE^\MM(G/H)$, over the Eilenberg-Mac~Lane space $\BB\chi(\upi(G/H),1)$, which is a $K(\upi(G/H),1)$ (cf. Theorem \ref{thnoneq}). In this section we prove that these parametrized $\Omega$-prespectra for different subgroups $H$ are the fixed point spectra of a parametrized $G$-$\Omega$-prespectrum $J_{G}\MM$ over the Eilenberg-Mac~Lane $G$-space $K_G(\upi,1)=\Psi \BB\chi_G(\upi,1)$ indexed over a trivial $G$-universe (that is, a naive parametrized $G$-$\Omega$-prespectrum). Here $\Psi\colon \OG\mbox{-}\Top \rightarrow G\mbox{-}\Top$ is Elmendorf's functor (cf. \cite{elm}), right adjoint to the functor $\Phi$ (cf. Equation \ref{eqphi}). Notice that since we are considering cohomology theories from arbitrary coefficient systems, we will not have the required transfer maps to form a spectrum indexed over a complete $G$-universe.

For each $n$ we apply Elmendorf's functor $\Psi$ to $\EE^\MM_n$ to obtain an ex-$G$-space $L_{\upi}(M,n) \rightarrow K_G(\upi,1)$. We denote this ex-$G$-space by $J_G \MM_n.$ 

\begin{prop}
The ex-$G$-spaces $J_G \MM=\{J_G \MM_n\}$ form an $\Omega$-$G$-prespectrum over $K_G(\upi,1)$ indexed on a trivial $G$-universe.
\end{prop}
\begin{proof}
 We have to check that the prespectrum is level-$qf$-fibrant and that the structure maps $J_G \MM_n\rightarrow \Omega_{K_G(\upi,1)}J_G \MM_{n+1}$ are $q$-equivalences of ex-$G$ spaces over $K_G(\upi,1)$.

Since $q$-equivalences of ex-$G$ spaces is just a $G$-homotopy equivalence of the total $G$-spaces, to prove the latter fact it is enough to check that for every subgroup $H$, the induced map on homotopy groups $\pi^H_*J_G \MM_n \rightarrow \pi^H_*\Omega_{K_G(\upi,1)}J_G \MM_{n+1}$ is an equivalence. That is,
$$\pi_*J_G \MM^H_n \cong \pi_*\Omega_{K(\upi(G/H),1)}J_G \MM^H_{n+1}.$$
We use $J_G \MM_n^H\simeq \EE^\MM_n(G/H)$ from the definition of Elmendorf's construction, so the result follows from the fact that $\EE^\MM(G/H)$ is a parametrized $\Omega$- prespectrum.

To show that the prespectrum is level $qf$-fibrant, we need to check that $J_G\MM_n \rightarrow K_G(\upi,1) $ is a $qf$-fibrant ex-$G$-space. By definition \cite[Definition 7.2.7 and Remark 7.2.11]{par}  this is true if $\Maps_G(S,J_G(\MM)_n) \rightarrow \Maps_G(S,K_G(\upi,1))$ is $qf$-fibrant non equivariantly for every finite $G$ set $S$. Since as $G$-spaces these are disjoint union of orbits $G/H $ it suffices to check that $J_G(\MM)_n^H \rightarrow K_G(\upi,1)^H$ is a Serre fibration and so a $qf$-fibration. But we have already proved this in Proposition \ref{padjoint}. Thus $J_G(\MM)$ is a parametrized $G$-$\Omega$-prespectrum over $K_G(\upi,1)$.
\end{proof}

\begin{defn}\cite[Definition 21.2.2]{par}
 For naive $G$-spectra $J$ and $E$ over $B$, define the \emph{$J$-cohomology groups} of $E$ by $J_G^n(E):=\pi_{-n}^G(r_*F_{B}(E, J))$
\end{defn}
This can be rewritten as
$$J_G^n(E)=[S_{B}^{-n}, F_{B}(E, J)]_{G,B}\cong [E,F_{B}(S_{B}^{-n}, J)]_{G,B} \cong [E,\Sigma_{B}^n J]_{G,B}.$$

We now show that Bredon cohomology with local coefficients can be described as the cohomology theory defined by the naive $G$-spectrum $J_G\MM$. As in the non-equivariant case, we fix an $\upi$-module $\MM=(M,\uphi)$. Let $X$ be a $G$-CW complex over $K_G(\upi,1)$, given by a map $\theta\colon X\to K_G(\upi,1)$. This gives an equivariant local coefficient system $\theta^*\MM$ on $X$. We make $X$ into an ex-$G$-space by adding a disjoint basepoint, $X_{+K_G(\upi,1)}=X\coprod K_G(\upi,1)$.
\begin{thm}
\label{thm:tgspec}
With notation as above, $$H^n_G(X;\theta^*\MM) \cong J_G \MM^n(X_{+K_G(\upi,1)}).$$
\end{thm}
\begin{proof}
By definition of the cohomology theory defined by the parametrized $G$-spectrum $J_G\MM$, we have
\begin{align}
J_G \MM^n(\Sigma^\infty_{K_G(\upi,1)} X_{+K_G(\upi,1)})
                              &= [\Sigma^\infty_{K_G(\upi,1)} X_{+K_G(\upi,1)},\Sigma_{K_G(\upi,1)}^n J_G \MM]_{K_G(\upi,1)}\nonumber \\
                              &= [X_{+K_G(\upi,1)}, \Omega^\infty_{K_G(\upi,1)} \Sigma_{K_G(\upi,1)}^n J_G \MM]_{K_G(\upi,1)} \nonumber \\
                              &= [X_{+K_G(\upi,1)}, J_G \MM_n]_{K_G(\upi,1)} \nonumber \\
                              &= [X_{+K_G(\upi,1)},\Psi\EE^{\MM}_n]_{K_G(\upi,1)} \nonumber \\
                              &= [X_{+K_G(\upi,1)},L_{\upi}(M,n)]_{K_G(\upi,1)} \nonumber
\end{align}
\noindent From Theorem \ref{thmol}, we know that the last expression is Bredon cohomology with local coefficients.
\end{proof}

\providecommand{\bysame}{\leavevmode\hbox to3em{\hrulefill}\thinspace}
\providecommand{\MR}{\relax\ifhmode\unskip\space\fi MR }
\providecommand{\MRhref}[2]{%
  \href{http://www.ams.org/mathscinet-getitem?mr=#1}{#2}
}
\providecommand{\href}[2]{#2}


\begin{thebibliography}{BFGM03}

\bibitem{bla}
A.~L.~ Blakers, Some relations between homology and homotopy groups,  \emph{Ann. of. Math} \textbf{49} (1948), no.~2, 428--461. \MR{0024132 (9,457b)}

\bibitem{br}
G.~E.~Bredon, \emph{Equivariant cohomology theories}, Lecture Notes in
  Mathematics, No. 34, Springer-Verlag, Berlin, 1967. \MR{0214062 (35
  \#4914)}

\bibitem{rbrown1}
R.~Brown, Exact sequences of fibrations of crossed complexes,
  homotopy classification of maps, and nonabelian extensions of groups, \emph{J.
  Homotopy Relat. Struct.} \textbf{3} (2008), no.~1, 331--342. \MR{2426184
  (2009e:18024)}

\bibitem{bfg}
M.~Bullejos, E.~Faro, and M.~A. Garc{\'i}a-Mu{\~n}oz, Homotopy colimits
  and cohomology with local coefficients, \emph{Cah. Topol. G{\'e}om. Diff{\'e}r.
  Cat{\'e}g.} \textbf{44} (2003), no.~1, 63--80. \MR{1961526 (2003j:18017)}

\bibitem{rbrown4}
R.~Brown and M.~Golasi{\'n}ski, A model structure for the homotopy
  theory of crossed complexes, \emph{Cah. Topol. G{\'e}om. Diff{\'e}r.
  Cat{\'e}g.}\textbf{30} (1989), no.~1, 61--82. \MR{1000831 (90f:18012)}

\bibitem{rbrown2}
R.~Brown, M.~Golasi{\'n}ski, T.~Porter and A.~Tonks, Spaces of maps into
  classifying spaces for equivariant crossed complexes, \emph{Indag. Math. (N.S.)}
  \textbf{8} (1997), no.~2, 157--172. \MR{1621979 (99j:18015)}

\bibitem{rbrown}
R.~Brown, M.~Golasi{\'n}ski, T.~Porter and A.~Tonks, Spaces of maps into classifying spaces for equivariant crossed
  complexes. {II}. {T}he general topological group case, \emph{$K$-Theory}
  \textbf{23} (2001), no.~2, 129--155. \MR{1857078 (2002j:18015)}

\bibitem{rbrown5}
R.~Brown and P.~J.~ Higgins, On the algebra of cubes, \emph{J. Pure. Appl. Algebra}
\textbf{21} (1981), no.~3, 233--260. \MR{0617135 (82m:55015a)}

\bibitem{rbrown5a}
R.~Brown and P.~J.~ Higgins, Colimit theorems for relative homotopy groups, \emph{J. Pure. Appl. Algebra}
\textbf{22} (1981), no.~1, 11--41. \MR{0617135 (82m:55015b)}


\bibitem{rbrown3}
R.~Brown and P.~J.~ Higgins, The classifying space of a crossed complex, \emph{Math. Proc.
  Cambridge Philos. Soc.} \textbf{110} (1991), no.~1, 95--120. \MR{1104605
  (92b:55024)}

\bibitem{brownb}
R.~Brown, P.~J.~Higgins and R.~Sivera, \emph{Nonabelian algebraic
  topology}, EMS Tracts in Mathematics, vol.~15, European Mathematical Society
  (EMS), Z{\"u}rich, 2011, Filtered spaces, crossed complexes, cubical homotopy
  groupoids, With contributions by Christopher D. Wensley and Sergei V.
  Soloviev. \MR{2841564}
%
\bibitem{cp}
   J.M.~Cordier and T.~Porter,
  Homotopy coherent category theory,
   \emph{Trans. Amer. Math. Soc.},\textbf{349} (1997), no.~1, 1--54. \MR{1376543 (97d:55032)}
%
\bibitem{elm}
A.~D.~Elmendorf, Systems of fixed point sets, \emph{Trans. Amer. Math. Soc.}
  \textbf{277} (1983), no.~1, 275--284. \MR{690052 (84f:57029)}

\bibitem{git}
S.~Gitler, Cohomology operations with local coefficients, \emph{Amer. J.
  Math.} \textbf{85} (1963), 156--188. \MR{0158398 (28 \#1621)}

\bibitem{gj}
P.~G.~Goerss and J.~F.~Jardine, \emph{Simplicial homotopy theory}, Progress
  in Mathematics, vol. 174, Birkh{\"a}user Verlag, Basel, 1999. \MR{1711612
  (2001d:55012)}

\bibitem{hir}
Y.~Hirashima, A note on cohomology with local coefficients, \emph{Osaka
  J. Math.} \textbf{16} (1979), no.~1, 219--231. \MR{527027 (80f:55008)}


\bibitem{hue1}
J.~Huebschmann, Sur les premi\`eres diff\'erentielles de la suite spectrale
              cohomologique d'une extension de groupes, \emph{Osaka
  J. Math.C. R. Acad. Sci. Paris S\'er. A-B} \textbf{285} (1977), no.~15, A929--A931. \MR{0472970 (57 \#12649)}

\bibitem{hue2}
J.~Huebschmann, Crossed n-fold extensions of groups and cohomology, \emph{Comment. Math. Helv.} \textbf{55} (1980), no.~2, 302--313. \MR{0576608 (82e:20063)}


\bibitem{may}
J.~P.~May, \emph{Simplicial objects in algebraic topology}, Van Nostrand
  Mathematical Studies, No. 11, D. Van Nostrand Co., Inc., Princeton,
  N.J.-Toronto, Ont.-London, 1967. \MR{0222892 (36 \#5942)}

\bibitem{ala}
J.~P.~May, \emph{Equivariant homotopy and cohomology theory}, CBMS Regional
  Conference Series in Mathematics, vol.~91, Published for the Conference Board
  of the Mathematical Sciences, Washington, DC, 1996, With contributions by M.
  Cole, G. Comeza{\~n}a, S. Costenoble, A. D. Elmendorf, J. P. C. Greenlees, L.
  G. Lewis, Jr., R. J. Piacenza, G. Triantafillou, and S. Waner. \MR{1413302
  (97k:55016)}

\bibitem{par}
J.~P.~May and J.~Sigurdsson, \emph{Parametrized homotopy theory}, Mathematical
  Surveys and Monographs, vol. 132, American Mathematical Society, Providence,
  RI, 2006. \MR{2271789 (2007k:55012)}
%
 \bibitem{m}
 J.~M{\o}ller,
On equivariant function spaces,
  \emph{Pacific J. Math.}
   \textbf{142} (1990), no.~1, 103--119. \MR{1038731 (91a:55024)}
%
 
\bibitem{mm}
A.~Mukherjee and G.~Mukherjee, Bredon-{I}llman cohomology with
  local coefficients, \emph{Quart. J. Math. Oxford Ser.} (2) \textbf{47} (1996),
  no.~186, 199--219. \MR{MR1397938 (98m:55008)}
%
\bibitem{ms}
G.~Mukherjee and D.~Sen, Equivariant simplicial cohomology with
  local coefficients and its classification, \emph{Topology Appl.} \textbf{157}
  (2010), no.~6, 1015--1032. \MR{2593715 (2011f:55037)}
%
\bibitem{qui}
D.~G.~Quillen, \emph{Homotopical algebra}, Lecture Notes in Mathematics,
  No. 43, Springer-Verlag, Berlin, 1967. \MR{MR0223432 (36 \#6480)}

%
\bibitem{dieck}
T.~tom~Dieck, \emph{Transformation groups}, de Gruyter Studies in
  Mathematics, vol.~8, Walter de Gruyter \& Co., Berlin, 1987. \MR{889050
  (89c:57048)}

\bibitem{white1}
 J.~H.~C~ Whitehead, Combinatorial homotopy. {II}, \emph{Bull. Amer. Math. Soc.}
   \textbf{55} (1949), 453--496. \MR{0030760 (11,48c)}



\end{thebibliography}
\end{document}